\documentclass[11pt, twoside, reqno]{amsart}
\usepackage{graphicx} 
\usepackage{geometry}
\usepackage[english]{ babel}
\usepackage{todonotes}

\usepackage[utf8]{inputenc}
\usepackage{babel}
\usepackage{lmodern}

\usepackage{amsmath}
\usepackage{amssymb}
\usepackage{amsthm}

\usepackage{csquotes}

\usepackage{mathtools}
\usepackage{stmaryrd}

\usepackage{wrapfig}
\usepackage{subcaption}
\usepackage{multirow}

\usepackage{enumitem}
\usepackage{verbatim}
\usepackage{wrapfig}
\usepackage{makecell}
\usepackage{xcolor}

\usepackage[
backend=biber,
giveninits=true,
maxbibnames=99,
doi=true,
eprint=true,
url=true
]{biblatex}

\AtEveryBibitem{%
	\iffieldundef{doi}
	{\iffieldundef{eprint}
		{}
		{\clearfield{url}}}
	{\clearfield{url}}%
}
\usepackage{hyperref}
\hypersetup{colorlinks,linkcolor={blue!75!black},citecolor={blue!75!black},urlcolor={blue!75!black}}

\usepackage{biblatex}

\usepackage{tikz}
\definecolor{mypurple}{RGB}{140,80,200}
\definecolor{myorange}{RGB}{255,150,60}

\DeclareMathOperator*{\id}{id}

\DeclareMathOperator*{\sgn}{sgn}

\DeclareMathOperator*{\pos}{pos}
\DeclareMathOperator*{\lin}{lin}

\DeclareMathOperator*{\inter}{int}

\DeclareMathOperator*{\conv}{conv}

\DeclareMathOperator{\vol}{vol}

\newcommand{\R}{\mathbb{R}}

\numberwithin{equation}{section}

\usepackage[english]{ babel}

\usepackage[utf8]{inputenc}
\usepackage[T1]{fontenc}
\usepackage{lmodern}
\usepackage{verbatim}

\usepackage{amsmath}
\usepackage{amssymb}
\usepackage{amsthm}

\usepackage{csquotes}

\usepackage{mathtools}
\usepackage{stmaryrd}

\usepackage{wrapfig}
\usepackage{subcaption}
\usepackage{multirow}

\usepackage{sidenotes}
\usepackage{enumitem}
\usepackage{verbatim}
\usepackage{wrapfig}
\usepackage{makecell}
\usepackage{xcolor}
\usepackage{todonotes}

\usepackage{hyperref}
\hypersetup{backref, pdfpagemode=FullScreen, colorlinks=true,
  citecolor=magenta, linkcolor=cyan, urlcolor=blue}

\usepackage{csquotes} 

\makeatletter
\def\blx@addquote#1{%
  #1%
}
\DeclareFieldFormat{title}{#1}
\DeclareFieldFormat{citetitle}{#1}
\DeclareFieldFormat{booktitle}{#1}
\DeclareFieldFormat{maintitle}{#1}
\makeatother

\renewbibmacro{in:}{}

\DeclareFieldFormat{pages}{#1}

\DeclareFieldFormat{journaltitle}{\mkbibemph{#1}}
\DeclareFieldFormat{series}{\mkbibemph{#1}}
\DeclareFieldFormat{booktitle}{\mkbibemph{#1}} 

\usepackage{xurl} 
\usepackage{hyperref} 
\hypersetup{
    colorlinks=true,
    linkcolor=blue,
    citecolor=blue,
    urlcolor=blue,
    filecolor=magenta,
    pdfpagelayout=OneColumn,
    pdftitle={Your Paper Title},
    pdfauthor={Your Name},
    pdfsubject={Mathematics},
    pdfkeywords={keyword1, keyword2}
}

\usepackage{tikz}

\mathtoolsset{showonlyrefs}

\newcommand{\sph}{\mathbb{S}}
 
\newcommand{\cv}{{\tt cv}}

\DeclareMathOperator{\relint}{relint}
\let\emptyset\varnothing

\mathtoolsset{showonlyrefs}

\newtheorem{theorem}{Theorem}[section]
\newtheorem{proposition}[theorem]{Proposition}
\newtheorem{lemma}[theorem]{Lemma}
\newtheorem{conjecture}[theorem]{Conjecture}

\newtheorem{oldtheorem}{Theorem}

\newtheorem{oldconjecture}{Conjecture}

\newtheorem{corollary}[theorem]{Corollary}

\theoremstyle{definition}

\theoremstyle{remark}
\newtheorem{remark}[theorem]{Remark}
\newtheorem{example}[theorem]{Example}

\newcommand{\U}{\mathcal{U}}

\title{On the Solution Sets of the Cone-Volume Map}
\author{Tom Baumbach}
\email{baumbach@math.tu-berlin.de}
\address{Institut für Mathematik, Technische Universität Berlin, Sekr MA 4-1, Straße des 17 Juni 136, D-10632 Berlin, Germany}
\author{Zhang Chen}
\email{2111165@tongji.edu.cn}
\address{School of Mathematical Sciences, Key Laboratory of Intelligent Computing and Applications (Ministry of Education), Tongji University, Shanghai, 200092, China} 
\author{Martin Henk} 
\address{Institut für Mathematik, Technische Universität Berlin, Sekr MA 4-1, Straße des 17 Juni 136, D-10632 Berlin, Germany}
\email{henk@math.tu-berlin.de}
\date{\today}

\begin{document}

\begin{abstract}
For a fixed set $U=(u_1,\ldots,u_m)$ of outer unit normals, we consider the
cone-volume map which associates to a vector of support numbers $b\in\R^m_{> 0}$ the
cone-volume vector $\gamma\in\R^m$ of the polytope
\[
P(U,b)=\{x\in\mathbb{R}^n:U^\top x\leq b\}.
\]
We study its fibers $S(U,\gamma)$ over prescribed cone-volume vectors
$\gamma$, with particular emphasis on finiteness and non-uniqueness.
We give characterizations of the irreducibility of $U$ in terms of
properties of the cone-volume map and its solution sets. In particular,
for irreducible $U$, we show that $S(U,\gamma)$ is finite for all
$\gamma$ outside a set of measure zero in the cone-volume cone, and we
give an explicit upper bound for its cardinality. In the planar case,
we prove that $S(U,\gamma)$ has semialgebraic dimension at most one for
arbitrary $U$. Moreover, for every $U\in\mathcal{U}(2,m)$ with $m\geq 4$,
there exists a cone-volume vector $\gamma$ for which $S(U,\gamma)$
contains at least two distinct elements. As a consequence, there exist
two distinct smooth planar convex bodies with the same cone-volume
measure. Finally, for $U$ in general position, we show that the number of
solutions is odd for every cone-volume vector outside the exceptional set.
\end{abstract}

\maketitle
\section{Introduction}

The logarithmic Minkowski problem asks for the characterization of
Borel measures on the unit sphere that arise as cone-volume measures
of convex bodies, usually under the assumptions that $K$ is origin-symmetric or  centred. For a convex body $K\subset\R^n$ containing the
origin in its interior, its cone-volume measure $V_K$ is given by
\[
dV_K=\frac{1}{n}h_K\,dS_K,
\]
where $h_K$ denotes the support function of $K$ and $S_K$ its surface
area measure (see Section \ref{sec:preliminaries} for precise definitions). The logarithmic Minkowski problem was introduced and developed in
the groundbreaking paper \cite{boroczky2013logarithmic_JAMS}. It has become a central problem in modern
Brunn--Minkowski theory, with connections to affine convex geometry,
fully nonlinear Monge--Amp\`ere equations, functional inequalities, and
spectral problems. For its development and the surrounding
$L_p$-Brunn--Minkowski theory, we refer to
\cite{boroczky2013logarithmic_JAMS,Boeroeczky2023,
	KolesnikovMilman2022,HuangYangZhang2025,
	BoeroeczkyFigalliRamos2026, Schneider2014, boroczky2012log, CHEN2020107166}
and the references therein.

In this paper we study for a fixed collection
of outer unit normals, the set of polytopes having a
prescribed cone-volume vector (measure), with particular emphasis on the
dimension and cardinality of this solution set. To be more precise we need some notation. 
Throughout, we work in the $n$-dimensional Euclidean space $\R^n$.
We denote by $\langle\cdot,\cdot\rangle$ the standard scalar product,
by $\|\cdot\|$ the associated Euclidean norm, and by
\[
\sph^{n-1}:=\{x\in\R^n:\|x\|=1\}
\]
the Euclidean unit sphere. For a $k$-dimensional measurable set $M$
contained in an affine subspace of $\R^n$, $\vol_k(M)$ denotes its
$k$-dimensional Lebesgue measure with respect to its affine hull; in the case $k=n$ we just write $\vol(M)$. 
For $M\subset\R^n$, we write $\conv(M)$, $\pos(M)$ and $\lin(M)$ for its convex, positive
and linear hull, respectively.

For $n,m\in\mathbb N$, let
\[
\U(n,m)
\coloneqq 
\left\{
(u_1,\ldots,u_m)\in(\sph^{n-1})^m :
\pos(u_1,\ldots,u_m)=\R^n,\;
u_i\neq u_j \text{ for } i\neq j
\right\},
\]
and for $ U=(u_1,\ldots,u_m)\in\U(n,m)$ and  $b\in\R^m_{>0}$, we consider the polytope 
\[
P(U,b)\coloneqq 
\{x\in\R^n:U^\intercal x\leq b\}.
\]

Writing
\[
F_i(U,b)
=
P(U,b)\cap
\{x\in\R^n:\langle u_i,x\rangle=b_i\},
\]
the cone-volume measure of $P(U,b)$ on $\sph^{n-1}$ is defined by 
\[
V_{P(U,b)}(\cdot)
\coloneqq 
\sum_{i=1}^m
\frac{b_i}{n}\vol_{n-1}(F_i(b))\,\delta_{u_i}(\cdot),
\]

where $\delta_{u_i}$ is the Dirac measure centred at $u_i$, i.e., $\delta_{u_i}(\eta)=1$ if $u_i\in\eta $, otherwise $\delta_{u_i}(\eta)=0$, and here $\eta\subset\sph^{n-1}$ is a Borel set.  
Accordingly, we define the \emph{cone-volume vector}
\[
\gamma(U,b)
=
\left(
\frac{b_1}{n}\vol_{n-1}(F_1(U,b)),\ldots,
\frac{b_m}{n}\vol_{n-1}(F_m(U,b))
\right)
\]
consisting of the $n$-dimensional volumes of the cones $\conv(F_i(U,b)\cup\{0\})$, $1\leq i\leq m$. Observe, that some $F_i(U,b)$ can be empty or of dimension less than $n-1$.

The  \emph{cone-volume map} is the map assigning to a right hand side $b\in \R^m_{>0}$ the cone-volume vector of the polytope $P(U,b)$ , i.e., 
\[
f_{\cv}^U:\R^m_{>0}\longrightarrow\R^m_{\geq0},
\qquad
b\longmapsto\gamma(U,b).
\]
When there is no ambiguity, we suppress the dependence on $U$ and write
$P(b)$, $F_i(b)$, and $\gamma(b)$ instead of
$P(U,b)$, $F_i(U,b)$, and $\gamma(U,b)$. We
further set
\[
C_{\cv}(U)
\coloneqq 
f_{\cv}^U(\R^m_{>0})\cap\R^m_{>0}
\]
and, for $\gamma\in C_{\cv}(U)$,
\[
S(U,\gamma)
\coloneqq 
(f_{\cv}^U)^{-1}(\gamma).
\]
Thus $S(U,\gamma)$ is the fiber of the cone-volume map over $\gamma$ which is known to be a semialgebraic set (see Section 2).
The main purpose of this paper is to study the structure of these
fibers, in particular their dimension and cardinality.

An important role is played by the reducibility of $U$. We call $U\in\U(n,m)$ \emph{reducible} if there
exists a nonempty proper subtuple $U_1\subsetneq U$ such that
\[
\R^n
=
\lin(U_1)\oplus\lin(U\setminus U_1),
\]
and \emph{irreducible} otherwise. 
In the geometric setting this is equivalent to a decomposition 
\[
P(U,b)=Q_1\oplus Q_2, 
\]
where $Q_1$ and $Q_2$ are polytopes of dimensions $n_1,n_2 \geq 1$, respectively,  contained in complementary subspaces. It is not hard to see that for $t>0$ the polytopes $t^{n_2}\,Q_1\oplus t^{-n_1}\,Q_2$ have the same cone-volume vector as $P(U,b)$. It is conjectured that this is essentially the only source of non-uniqueness in the centred logarithmic Minkowski problem.  
We recall that a convex body is called centred if its centroid is the origin.

\begin{oldconjecture}[
	{\cite[Conjecture 3.1]{Boeroeczky2023}}]
	\label{conj:centered-uniqueness}
	Let $K,M\subset\R^n$ be centred convex bodies with $\vol(K)=\vol(M)$. Then $V_K=V_M$ if and
	only if there exist convex bodies
	$K_1,\ldots,K_k,M_1,\ldots,M_k$ such that
	\[
	K=K_1\oplus\cdots\oplus K_k,
	\qquad
	M=M_1\oplus\cdots\oplus M_k,
	\]
	with $\dim K_i\geq1$ and  $K_i=\lambda_iM_i$, for some $\lambda_i>0$, $1\leq i\leq k$.
\end{oldconjecture}

In particular, Conjecture~\ref{conj:centered-uniqueness} predicts
uniqueness for centred polytopes with irreducible outer unit normals.
The conjecture is known in the origin-symmetric planar case
\cite{boroczky2012log}; see also \cite{CHEN2020107166} for
origin-symmetric bodies close to the Euclidean ball. To formulate the planar symmetric result in terms of the fibers
introduced above, we call
$U=(u_1,\ldots,u_{2m})\in\U(n,2m) $
\emph{origin symmetric} if, after a suitable relabelling, for $i=1,\ldots,m$  

\begin{equation*} u_{i+m}=-u_i.
\end{equation*} 	

Observe, for an origin-symmetric $U\in\U(2,2m)$, irreducibility is equivalent
to $m>2$.

\begin{oldtheorem}[\cite{boroczky2012log}]
	\label{thm:known-symmetric-plane}
	Let $U\in\U(2,2m)$ be origin symmetric and irreducible, i.e.,
	$m>2$, and let
	\[
	\gamma=f_{\cv}^U(b,b)\in C_{\cv}(U)
	\]
	for some $b\in\R^m_{>0}$. Then
	\[
	\left|
	S(U,\gamma)
	\cap
	\bigl\{(c,c):c\in\R^m_{>0}\bigr\}
	\right|
	=1.
	\]
\end{oldtheorem}

Thus Theorem~\ref{thm:known-symmetric-plane} gives uniqueness within
the origin-symmetric part of the fiber. Without a symmetry or
centring assumption, however, uniqueness can fail even when the outer
unit normals are irreducible. Chen, Li and Zhu
\cite{ChenLiZhu2019} constructed such examples in every dimension.

\begin{oldtheorem}[\cite{ChenLiZhu2019}]
	\label{thm:known-chen-li-zhu}
	For every $n\geq2$ there exist an irreducible
	$U\in\U(n,2^n+1)$ and a vector $\gamma\in C_{\cv}(U)$ such that
	\[
	|S(U,\gamma)|\geq2.
	\]
\end{oldtheorem}

The irreducibility of $U$ is not explicitly part of the statement in
\cite{ChenLiZhu2019}, but follows from the construction.

The planar discrete problem was first studied by Stancu
\cite{Stancu2002,Stancu2003}, who approached it through homothetic
solutions of crystalline curvature flows. In \cite{Stancu2002} she
established the existence of solutions under suitable assumptions on the
data, and \cite{Stancu2003} is devoted to the number of solutions, which
is also the subject of the present paper. In the plane,
the situation is completely understood for quadrilaterals by the work
of Liu et al.~\cite{LiuLuSunEtAl2024}; see also
\cite{Subspace_Concentration_of_Geometric_Measures} for the trapezoid
case.

\begin{oldtheorem}[\cite{LiuLuSunEtAl2024}]
	\label{thm:known-quadrilaterals}
	Let $U\in\U(2,4)$ and $\gamma\in C_{\cv}(U)$. Then
	\begin{enumerate}
		\item $S(U,\gamma)$ is finite if and only if $U$ is irreducible;
		\item if $U$ is irreducible, then
		$
		|S(U,\gamma)|\in\{1,2,3\}$.
	\end{enumerate}
\end{oldtheorem}
 More precise
criteria distinguishing the possible numbers of solutions in
Theorem~\ref{thm:known-quadrilaterals} will be recalled in
Section~\ref{sec:planar}.

We note that \cite[Theorem~1.3]{Stancu2003} also asserts uniqueness
without a symmetry assumption, in particular whenever $m\geq4$ and the
vectors of $U$ are pairwise linearly independent. This does not hold in
that generality: the quadrilaterals in
Theorem~\ref{thm:known-quadrilaterals}, Example~\ref{exp:pentagon} and
Theorem~\ref{thm:1.7} provide counterexamples. The same applies to the
dichotomy in \cite[Section~3]{Stancu2003}, according to which the
problem has either exactly one or uncountably many solutions.

We believe that the first part of Theorem~\ref{thm:known-quadrilaterals} holds true in general.

\begin{conjecture}
	\label{conj:irreduc}
	Let $U\in\U(n,m)$ and $\gamma\in C_{\cv}(U)$. Then 
	$S(U,\gamma)$ is finite if and only if $U$ is irreducible.
\end{conjecture}

This conjecture is closely related to the dimension conjecture for
cone-volume fibers formulated in
\cite[Conjecture 5.2]{baumbach2025polynomialinequalitiesconevolumespolytopes}.
In fact, the two questions are equivalent in the following sense.

\begin{proposition}
	\label{thm:finiteness-dimension}
	Let $U\in\U(n,m)$ have irreducible decomposition
	\[
	U=S_1\cup\cdots\cup S_k.
	\]
	Then the following are equivalent:
	\begin{enumerate}
		\item for every $j=1,\ldots,k$ and every
		$\eta\in C_{\cv}(S_j)$, the fiber $S(S_j,\eta)$ is finite;
		
		\item for every $\gamma\in C_{\cv}(U)$,
		\[
		\dim^* S(U,\gamma)=k-1.
		\]
	\end{enumerate}
\end{proposition}

Here $\dim^* X$ denotes semialgebraic dimension; see end of Section \ref{sec:preliminaries}. In particular, a
semialgebraic set has semialgebraic dimension zero if and only if it is finite.

Our first main non-uniqueness result shows that, in the plane, no
condition on the outer unit normals alone can guarantee uniqueness
once at least four normals are prescribed.

\begin{theorem}
	\label{thm:1.7}
	Let $m\geq4$ and $U\in\U(2,m)$. Then there exists
	$\gamma\in C_{\cv}(U)$ such that
	\[
	|S(U,\gamma)|\geq2.
	\]
\end{theorem}

In particular, non-uniqueness is not restricted to exceptional
configurations of outer unit normals. Our proof is based on a
stability argument showing that non-uniqueness can be preserved when
additional outer unit normals are inserted. Iterating this
construction and passing to a limit yields a smooth counterpart.

\begin{corollary}
	\label{thm:1.8}
	There exist two distinct smooth convex bodies
	$K_1,K_2\subset\R^2$ such that
	\[
	V_{K_1}=V_{K_2}.
	\]
\end{corollary}

Non-uniqueness for smooth data in the plane is known. Written in terms of
the support function, the planar logarithmic Minkowski problem is the
equation $h(h''+h)=f$ on $\sph^1$, whose solutions are precisely the
self-similar shrinking curves of the anisotropic curvature flow with
interfacial energy $f$. For a suitable non-symmetric energy, Yagisita
\cite{Yagisita2006} constructed two such curves, and thereby two distinct
smooth convex bodies with the same cone-volume measure. In
Corollary~\ref{thm:1.8}, smoothness is to be understood in the weaker
sense that every boundary point of $K_1$ and $K_2$ has a unique outer unit
normal; in fact the construction produces bodies whose cone-volume
measures have an atom at every direction of a dense subset of $\sph^1$.
Its point is rather that non-uniqueness for smooth bodies can be reached
from the discrete setting alone, by iterating the insertion argument of
Theorem~\ref{thm:inserting-normal}.

Our second main result concerns finiteness. Although the fibers need
not be singletons, for irreducible $U$ they are finite for almost
every cone-volume vector, with a bound depending only on $n$ and $m$.

\begin{theorem}
	\label{thm:generic-finiteness-intro}
	Let $U\in\U(n,m)$ be irreducible. There exists a set
	$B(U)\subset\R^m_{\geq0}$ of Lebesgue measure zero such that, for every
	$\gamma\in C_{\cv}(U)\setminus B(U)$,
	\[
	|S(U,\gamma)|
	<
	(2mn)^m m^{2mn}.
	\]
\end{theorem}

Consequently, Conjecture~\ref{conj:irreduc} holds for almost every
cone-volume vector. The proof of
Theorem~\ref{thm:generic-finiteness-intro} yields, more precisely, the
sharper estimate
\[
|S(U,\gamma)\cap A|
\leq n(2n-1)^{m-1}
\]
on each proper type-cone $A$; see
Section~\ref{sec:irreducible}. 

More could be said if $U$ is in general position. In this case, Theorem~\ref{thm:con} shows that the cardinality
$|S(U,\gamma)|$ is constant on every pathwise connected component of
$C_{\cv}(U)\setminus B(U)$. In addition, the cardinality is odd.

\begin{theorem}
	\label{thm:parity}
	Let $U\in\U(n,m)$ be in general position and 
	$\gamma\in C_{\cv}(U)\setminus B(U)$. Then $|S(U,\gamma)|$ is odd.
\end{theorem}

The proof combines the variational characterization of the solutions,
which goes back to \cite{boroczky2013logarithmic_JAMS,Zhu2014}, with a
degree argument; general position enters only through Zhu's compactness
theorem. Irreducibility alone does not suffice, as trapezoids show
(Remark~\ref{rem:parity-sharp}). As a consequence, for $U\in\U(2,m)$,
$m\geq4$, in general position, Theorem~\ref{thm:1.7} can be improved to
$|S(U,\gamma)|\geq3$ for some $\gamma$
(Remark~\ref{rem:parity-consequences}).

Finally, in Section~\ref{sec:derivative} we investigate the
differential structure of the cone-volume map. On each type-cone its
components are homogeneous polynomials, and in
Theorem~\ref{thm:dev} we derive an explicit formula for its
derivative. This is used to relate algebraic properties of the
cone-volume map to geometric properties of the prescribed outer unit
normals. For instance, Proposition~\ref{thm:gamma_linear_iff_zylinder}
characterizes when a cone-volume polynomial $f_i$ is linear in the
corresponding support number $b_i$, while
Proposition~\ref{ubd} gives a sufficient condition for $f_i$ to
contain the monomial $b_i^n$.

The paper is organized as follows.
Section~\ref{sec:preliminaries} collects the necessary preliminaries
on irreducible decompositions, type-cones and semialgebraic sets.
In Section~\ref{sec:planar} we study the planar case. We establish the
dimension bound for the fibers, develop a stability argument for
non-uniqueness under the insertion of additional outer unit normals,
and prove Theorem~\ref{thm:1.7} and Corollary~\ref{thm:1.8}.
Section~\ref{sec:irreducible} is devoted to irreducible sets of outer
unit normals in arbitrary dimension and contains the proof of Theorem~\ref{thm:generic-finiteness-intro} and of Theorem~\ref{thm:parity2}, which is a stronger version of Theorem~\ref{thm:parity}, together with further results on the
number of solutions. Section~\ref{sec:derivative} studies the
derivative of the cone-volume map and its relation to the geometry of
the prescribed outer unit normals.

\section{Preliminaries}
\label{sec:preliminaries}

This paper builds upon the work presented in \cite{baumbach2025polynomialinequalitiesconevolumespolytopes}. For the sake of completeness, we revisit its preliminaries and as a general reference for  basic notations regarding convex bodies or polytopes we refer to \cite[Ch.1, Ch.2]{Schneider2014}.
First we recall basic properties of the cone-volume set.

Recall that for $U\in\mathcal{U}(n,m)$ and  $b\in \R^m_{> 0}$, the cone-volume vector of $P=P(U,b)$ is
\begin{equation*}
	f_{\cv}(b)=\frac{1}{n}\Big(b_1\vol_{n-1}(F(u_1)),\dots, b_m\vol_{n-1}(F(u_m))\Big)\in \R^m_{\geq 0}. 
\end{equation*}
For $u\in\sph^{n-1}$ let 
\[
F_P(u)=P\cap \{ x\in\R^n: \langle u,x\rangle=h_P(u)\}.
\]
be the face of $P$ in direction $u$. As mentioned in the introduction, for a convex body $K$, the function  
\begin{equation*}
	h_K :\sph^{n-1}\to\R \text{ with }h_K(u)\coloneqq \max\{ \langle u,x\rangle : x \in K \}
\end{equation*} 
is called the support function of $K$. Apparently if $h_P(u_i)<b_i$, or in general, $\dim F_P(u_i)<n-1$ the ith entry of the cone-volume vector is 0.

The cone-volume cone $C_\cv(U)\coloneqq f_{\cv}(\mathbb R^m_{> 0})\cap \mathbb R_{>0}^m$ can be decomposed by considering the (in matroid theory so-called) \textit{separators} of $U$, denoted by $\mathcal F(U)$ hereafter:
\[
\mathcal F(U)\coloneqq \Bigl\{ S \subseteq U : 1 \leq  \mathrm{rk}(S) \leq n - 1, \  \R^n=\lin(S)\oplus\lin(U\setminus S) \Bigl\}.
\]
Here by $S \subseteq U$ we mean a subtuple of columns in $U$, and by $\mathrm{rk}(S)$ we denote the dimension of $\lin S$. As mentioned in the introduction  $U \in \mathcal{U}(n,m)$ is called irreducible if $\mathcal{F}(U) = \emptyset$, otherwise reducible.

Observe that if $S$ is a separator of $U$, so is $U\setminus S$. Looking at the separators of $S$ and $U\setminus S$ and so forth leads to a unique partition of $U$ into irreducible separators:
\[
U=S_1\cup\dots\cup S_k.
\]

We point out that such a decomposition has a background from matroid theory; see \cite[Theorem 6.81]{Aigner1997}. Using this, we deduce the following decomposition of the cone-volume set, where we use the normalized set 
\begin{equation*}
	C_{\cv}^1(U)\coloneqq \{\gamma\in C_{\cv}(U): \|\gamma\|_1=1\},
\end{equation*}
where $\|x\|_1=|x_1|+\cdots +|x_m|$ denotes the $1$-norm.

\begin{proposition}[{\cite[Proposition 2.7]{baumbach2025polynomialinequalitiesconevolumespolytopes}}]
	\label{PropositionSeparatorsPsccAndConeVolumesDimension}
	Let $U=S_1\cup\cdots\cup S_k$  
	be the (unique) decomposition of $U$ into irreducible separators. Then, up to a permutation of the coordinates,
	\[
	C_{\cv}^1(U)
	=
	\frac{\dim(S_1)}{n}C_{\cv}^1(S_1)
	\times\cdots\times
	\frac{\dim(S_k)}{n}C_{\cv}^1(S_k).
	\]
\end{proposition}

We call $U\in \mathcal{U}(n,m)$ in \emph{general position} if each subtuple $S\subset U$ of cardinality $n$ is linearly independent. 
  
Two polytopes $P$ and $Q$ are called
\emph{strongly isomorphic} if for every $u\in\sph^{n-1}$
\[
\dim F_P(u)=\dim F_Q(u).
\] 

For fixed $U$, the combinatorial structure of $P(U,b)$ is constant
on suitable polyhedral cones in the space of support vectors. 
\begin{lemma}[{\cite[Lemma 3.1]{baumbach2025polynomialinequalitiesconevolumespolytopes}, \cite[Section 5.1]{Schneider2014}}]
	\label{res:type_cones_and_volume_polynomials}
	Let $U\in\U(n,m)$. Then $\R^m_{\geq0}$ admits a finite polyhedral
	subdivision, denoted by $\mathcal A(U)$, with the following
	properties. For every $A\in\mathcal A(U)$:
	\begin{enumerate}
		\item for $b,\bar b\in\inter A$, the polytopes
		$P(U,b)$ and $P(U,\bar b)$ are strongly isomorphic;
		
		\item there exist polynomials
		\[
		v^A,f_1^A,\ldots,f_m^A\in\R[b_1,\ldots,b_m]
		\]
		of degree at most $n$ such that, for $b\in\inter A$,
		\[
		v^A(b)=\vol(P(U,b))
		\]
		and
		\[
		f_i^A(b)=V_{P(U,b)}(\{u_i\}),
		\qquad i=1,\ldots,m.
		\]
	\end{enumerate}
\end{lemma}

The elements of $\mathcal A(U)$ are called \emph{type-cones}. A
type-cone $A\in\mathcal A(U)$ is called \emph{proper} if for $b\in\inter A$ every $u_i$
determines a facet of $P(U,b)$, i.e., $\dim F_i(b)=n-1$, $i=1,\ldots,m$.

We call $v^A$ the \emph{volume polynomial} and
$f_1^A,\ldots,f_m^A$ the \emph{cone-volume polynomials} on $A$.
Writing
\[
f_{\cv}^A(b)=(f_1^A(b),\ldots,f_m^A(b)),
\] 
we denote by 
\[
{\det}^A(b)\coloneqq \det Df_{\cv}^A(b)
\]
the Jacobian determinant of the cone-volume map on $A$. Whenever there is no
ambiguity in the choice of the type-cone, we omit the superscript
$A$. Since the cone-volume map is piecewise polynomial with respect
to the finite subdivision $\mathcal A(U)$, every fiber
$S(U,\gamma)$ is semialgebraic.

For later use, we also record the adjacency structure of the facets, again following \cite[Section 5.1]{Schneider2014}. 
For $A\in\mathcal A(U)$, let
\[
J^A
\coloneqq 
\left\{
(i,j)\in\{1,\ldots,m\}^2:
\dim F_i(b)=\dim F_j(b)=n-1,\;
\dim(F_i(b)\cap F_j(b))=n-2
\right\},
\]
where $b\in\inter A$. This definition is independent of the choice
of $b\in\inter A$. We say that $u_j$ is \emph{adjacent} to $u_i$ on
$A$ if $(i,j)\in J^A$, and write
\begin{equation}
	\label{eq:nr1}
J^A(i)\coloneqq \{j:(i,j)\in J^A\}.
\end{equation} 

Let $\theta_{ij}$ denote the angle between $u_i$ and $u_j$. Whenever
$(i,j)\in J^A$, set
\begin{equation}
	\label{eq:nr2}
\nu_{i,j}
\coloneqq \frac{1}{\sin(\theta_{i,j})}\left(
 u_j-\cos(\theta_{i,j})u_i\right) \text{ and } 
h_{i,j}
\coloneqq 
\csc(\theta_{i,j}) b_j-\cot(\theta_{i,j})b_i.
\end{equation}
Denoting by $u_i^\perp$ the orthogonal complement of $\lin(u_i)$ the vector  $\nu_{i,j}\in u_i^\perp$ 
 is the outer unit normal of the facet $F_i(b)$ at
the $(n-2)$-face $F_i(b)\cap F_j(b)$, and
\begin{equation}
\label{eq:facetineq}
F_i(b)
=
b_i u_i+
\left\{
x\in u_i^\perp:
\langle x,\nu_{i,j}\rangle\leq h_{i,j}
\text{ for }j\in J^A(i)
\right\}.
\end{equation}

Finally, we recall some terminology from semialgebraic geometry; see
\cite[Section 2.8]{bochnak2013real}. A set $X\subset\R^m$ is called
\emph{semialgebraic} if it is a finite Boolean combination of sets
defined by polynomial equalities and inequalities.

The \emph{semialgebraic dimension} of a nonempty semialgebraic set
$X\subset\R^m$, denoted by $\dim^*X$, is the largest integer $d$
for which some coordinate projection of $X$ onto $\R^d$ has nonempty
interior; see \cite[Section~2.8]{bochnak2013real} and
\cite[Fact 3.3]{EleftheriouSavatovsky2020}.  In particular,
\begin{equation} 
\dim^*X=0
 \text{ if and only if } X\text{ is finite}.
 \label{eq:sadim}
\end{equation} 
A map $f:X\to Y$ between semialgebraic sets is called 
\emph{semialgebraic} if its graph
\[
\{(x,f(x)):x\in X\}
\]
is a semialgebraic subset of the corresponding Euclidean space.
 For polynomials $p_1,\ldots,p_k$, we denote their common real and
complex zero sets by
\[
\mathcal V(p_1,\ldots,p_k)
\quad\text{and}\quad
\mathcal V_{\mathbb C}(p_1,\ldots,p_k),
\]
respectively.

\section{Planar case}\label{sec:planar}
In this section, we consider "planar" sets  $U=(u_1,\dots,u_m)\in\mathcal{U}(2,m)$, and assume that $u_1,\dots,u_m$ are ordered counterclockwise. 
Since we are in the plane, there is only one proper type-cone $ A^*\in\mathcal A(U)$.

For $b\in\operatorname{int}A^*$ and $1\leq i\leq m$, the length of the edge
$F_i(b)$ is given by
\begin{align}
	\vol_1(F_i(b))
	&=h_{i,i-1}+h_{i,i+1} \notag\\
	&=-(\cot\theta_{i,i-1}+\cot\theta_{i,i+1})b_i
	+\csc\theta_{i,i+1}b_{i+1}
	+\csc\theta_{i,i-1}b_{i-1} \notag\\
	&=-A_i(U)b_i+B_i(U)b_{i+1}+C_i(U)b_{i-1},
	\label{eq:planar-edge-length}
\end{align}
where the indices are understood modulo $m$ and
\[
A_i(U)\coloneqq \cot\theta_{i,i-1}+\cot\theta_{i,i+1},\qquad
B_i(U)\coloneqq \csc\theta_{i,i+1},\qquad
C_i(U)\coloneqq \csc\theta_{i,i-1}.
\]
Hence
\begin{equation}
	f_i(b)=f_i^{A^*}(b)
	=\frac12\left(
	-A_i(U)b_i^2+B_i(U)b_i b_{i+1}+C_i(U)b_{i-1}b_i
	\right).
	\label{eq:planarconevolumes}
\end{equation}
When there is no ambiguity, we simply write $A_i,B_i,C_i$.
Since $0<\theta_{i,i+1}<\pi$, we have $B_i>0$. Thus the equation
$f_i(b)=\gamma_i$ can, whenever $b_i\neq0$, be solved uniquely for
$b_{i+1}$:
\begin{equation}
	b_{i+1}
	=\frac{A_i b_i-C_i b_{i-1}}{B_i}
	+\frac{2\gamma_i}{B_i b_i}.
	\label{eq:planar-recursion}
\end{equation}

\subsection{Dimension of the solution set}

The elementary recursion \eqref{eq:planar-recursion} will be used repeatedly
below. In particular, the equations
\[
f_i(b)=\gamma_i,\qquad i=1,\ldots,m-2,
\]
successively determine $b_2,\ldots,b_{m-1}$ from $b_m$ and $b_1$,
provided $b_1,\ldots,b_{m-2}$ are nonzero. This immediately yields the
following observation.

\begin{lemma}
	\label{lem:two-support-numbers}
	Let $U\in\U(2,m)$ and $\gamma\in C_{\cv}(U)$. Then the coordinate projection
	\[
	\pi_2:S(U,\gamma)\longrightarrow\R^2,\qquad
	b\longmapsto(b_m,b_1),
	\]
	is injective.
\end{lemma}

\begin{proof}
	Let $b,b'\in S(U,\gamma)$ satisfy
	\[
	(b_m,b_1)=(b'_m,b'_1).
	\]
	Since $b,b'\in\R^m_{>0}$, the recursion \eqref{eq:planar-recursion} applies.
	Starting with $i=1$, it determines $b_2$ uniquely from $b_m$ and $b_1$.
	Proceeding successively with $i=2,\ldots,m-2$ determines
	$b_3,\ldots,b_{m-1}$. Hence $b=b'$.
\end{proof}

Since $\pi_2$ is a semialgebraic injective map,
Lemma~\ref{lem:two-support-numbers} yields
\[
\dim^*S(U,\gamma)\leq2;
\]
see, e.g., \cite[Section~2.8]{bochnak2013real}. In Theorem  \ref{thm:planar_dim_leq_1} we will improve on this bound to conjectured one of $\dim^* S(U,\gamma)\leq 1$. First, however, we point out that the projection onto one coordinate   
\begin{equation*}
	\pi_1 : S(U,\gamma) \rightarrow \R, \quad b \mapsto b_1
\end{equation*}
is not injective, in general. 
\begin{example}\label{exp:pentagon}
	Let \[U=\begin{pmatrix}
			1& \cos\frac{2\pi}{5}&\cos\frac{4\pi}{5}&\cos\frac{6\pi}{5}&\cos\frac{8\pi}{5}\\
			0&\sin\frac{2\pi}{5}&\sin\frac{4\pi}{5}&\sin\frac{6\pi}{5}&\sin\frac{8\pi}{5}
		\end{pmatrix}\] be the outer unit normals of a regular pentagon (arranged counterclockwise), and let $\gamma=(4,4,1,1,4)$. The given cone-volume measure is reflectional symmetric with respect to the line $\mathrm{span}\{u_1\}$; however, the corresponding Minkowski problem has three solutions, one of which is reflectional symmetric, and the other two are mirror images of each other. As a consequence, there are two solutions to this problem with the same cone-volume measure and the same $b_1$.
	
	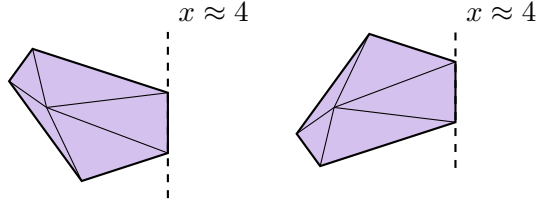
\begin{figure}[h!]
			\begin{center}
					\begin{tikzpicture}[scale=0.4]
							
							\definecolor{mypurple}{RGB}{150,100,210}
							
							\begin{scope}
									\fill (0,0) circle (2pt);
									
									\filldraw[mypurple!40, draw=black, line width=0.8pt]
									(1.162148599, -2.432962146) --
									(4.008549318, -1.508110942) --
									(4.008549318, 0.4876224531) --
									(-0.4605273728, 1.939712782) --
									(-1.237655842, 0.8700867407) -- cycle;
									
									\foreach \x/\y in {
											1.162148599/-2.432962146,
											4.008549318/-1.508110942,
											4.008549318/0.4876224531,
											-0.4605273728/1.939712782,
											-1.237655842/0.8700867407}{
											\draw[black, thin] (0,0) -- (\x,\y);
										}
									
									\draw[black, thick, dashed] (4.008, -3) -- (4.008, 2.5) node[black,above right] {$x \approx 4$};
								\end{scope}
							\begin{scope}[xshift=9.5cm] 
									\fill (0,0) circle (2pt);
									
									\filldraw[mypurple!40, draw=black, line width=0.8pt]
									(-1.237655842, -0.8700867407) --
									(-0.4605273728, -1.939712782) --
									(4.008549318, -0.4876224531) --
									(4.008549318, 1.508110942) --
									(1.162148599, 2.432962146) -- cycle;
									
									\foreach \x/\y in {
											-1.237655842/-0.8700867407,
											-0.4605273728/-1.939712782,
											4.008549318/-0.4876224531,
											4.008549318/1.508110942,
											1.162148599/2.432962146}{
											\draw[black, thin] (0,0) -- (\x,\y);
										}
									
									\draw[black, thick, dashed] (4.008, -2) -- (4.008, 2.5) node[black,above right] {$x \approx 4$};
								\end{scope}
						\end{tikzpicture}
					\caption{The non-reflectional-symmetric solutions to a reflectional symmetric cone-volume data}
				\end{center}
			
		\end{figure}
		The non-symmetric solutions to this problem, symbolically solved by OSCAR computer algebra system \cite{OSCAR-book,OSCAR}, are approximately $b=(4.01, 1.70, 1.51, 0.49, 2.67)$, and \newline $\overline{b}=(4.01, 2.67, 0.49, 1.51, 1.70)$. The entries are all algebraic numbers, and we list the minimal polynomial as following. 
	
	The minimal polynomial for $b_1$
	\begin{align*}
			&38416x^{32} - 16356200x^{28} + 1997820565x^{24} - 91998304800x^{20} + 1378110060800x^{16} \\
			- &1354809344000x^{12}
			- 16242442240000x^{8} - 821412495360000x^{4} + 687194767360000,
		\end{align*}
	the minimal polynomial for $b_2$ and $b_5$ is
	\begin{align*}
			&92236816x^{64} - 32002880980x^{60} + 422589265525x^{56} + 298932852687000x^{52} - 24329458160813650x^{48}\\
			+ &335963006048807500x^{44} + 72363603702201170125x^{40} - 6519700080313668200000x^{36} \\
			+ &297642512506948958720000x^{32}
			- 8678192734746686259200000x^{28} \\
			+ &166491642606924909772800000x^{24} - 2016627066901536702464000000x^{20}\\ 
			+ &14252971162702218002432000000x^{16} - 54240518635056864952320000000x^{12}\\
			+ &102379879969050748518400000000x^{8} - 90158461660255433523200000000x^{4}\\
			+ &29514790517935282585600000000,
		\end{align*}
	and the minimal polynomial for $b_3$ and $b_4$ is
	\begin{align*}
			&1475789056x^{64} - 401587802560x^{60} + 77767639351360x^{56} - 5844052978537900x^{52} + 166338239243277425x^{48}
			\\- &1501192555389456750x^{44} + 5392147886363919125x^{40} - 7072849133597453750x^{36}\\
			+ &1091546407030346250x^{32} + 117339582530156250x^{28} - 7969189581121875x^{24}\\
			- &603808300406250x^{20} + 21134757906250x^{16} + 630679218750x^{12}\\
			- &2139296875x^{8} - 77343750x^{4} + 390625.
		\end{align*}
\end{example}

In the planar case it is still possible to deduce that the dimension of $S(U,\gamma)$ 
is at most $1$, by considering the complex algebraic set $\mathcal{V}_{\mathbb{C}}(f_1 - \gamma_1,\dots,f_{m-1}(b) - \gamma_{m-1})$. Notice that we only consider the first $m-1$ polynomials for the dimension and not all $m$ equations. 
Additionally, we have no irreducibility assumption in the next statements. 

\begin{theorem}
	\label{thm:planar_dim_leq_1}
	Let $U\in\U(2,m)$ and $\gamma\in C_{\cv}(U)$. Then
	\[
	\dim^*S(U,\gamma)\leq1,
	\]
	and this bound is sharp.
\end{theorem}

For the proof we first study the complex zero set determined by the first
$m-2$ cone-volume equations.

\begin{lemma}
	\label{lem:Variety_of_cone_volumes}
	Let $U\in\U(2,m)$, $\gamma\in C_{\cv}(U)$, and let
	$f_i$, $1\leq i\leq m$, be as in \eqref{eq:planarconevolumes}. Then
	\[
	X_\gamma
	\coloneqq \mathcal V_{\mathbb C}
	(f_1-\gamma_1,\ldots,f_{m-2}-\gamma_{m-2})
	\subset\mathbb C^m
	\]
	is a smooth irreducible algebraic set of complex dimension $2$.
\end{lemma}

\begin{proof}
	For $b\in X_\gamma$ we have
	\[
	b_i\neq0,\qquad i=1,\ldots,m-2,
	\]
	since otherwise \eqref{eq:planarconevolumes} would give
	$f_i(b)=0\neq\gamma_i$. Hence \eqref{eq:planar-recursion} shows that every
	$b\in X_\gamma$ is uniquely determined by the two coordinates $b_m$ and
	$b_1$.
	
	Set
	\[
	\Omega_1\coloneqq \mathbb C^2,\qquad
	r_0(x,y)\coloneqq x,\qquad r_1(x,y)\coloneqq y,
	\]
	and, for $i=1,\ldots,m-2$, define recursively
	\[
	\Omega_{i+1}\coloneqq \{(x,y)\in\Omega_i:r_i(x,y)\neq0\}
	\]
	and, for $(x,y)\in\Omega_{i+1}$,
	\[
	r_{i+1}(x,y)
	\coloneqq \frac{A_i r_i(x,y)-C_i r_{i-1}(x,y)}{B_i}
	+\frac{2\gamma_i}{B_i r_i(x,y)}.
	\]
	By induction, $r_i$ is regular on $\Omega_i$, and therefore
	$\Omega_{i+1}$ is Zariski open in $\Omega_i$, hence also in $\mathbb C^2$.
	Moreover, all these sets are nonempty. Indeed, since
	$\gamma\in C_{\cv}(U)$, there exists $b=(b_1,\ldots,b_m)\in S(U,\gamma)$,
	and \eqref{eq:planar-recursion} shows that $(b_m,b_1)\in\Omega_i$ for every
	$i$. Thus
	\[
	\Omega\coloneqq \Omega_{m-1}
	\]
	is a nonempty Zariski-open subset of $\mathbb C^2$. Consequently,
	$\Omega$ is smooth and irreducible and
	\[
	\dim_{\mathbb C}\Omega=2.
	\]
	
	By \eqref{eq:planar-recursion}, the map
	\[
	\Phi:\Omega\longrightarrow X_\gamma,\qquad
	(x,y)\longmapsto
	\bigl(r_1(x,y),\ldots,r_{m-1}(x,y),r_0(x,y)\bigr)
	\]
	is well defined and regular. Conversely, if $b\in X_\gamma$, then
	$b_i\neq0$ for $i=1,\ldots,m-2$, and
	\eqref{eq:planar-recursion} shows recursively that
	\[
	\Phi(b_m,b_1)=b.
	\]
	Thus $\Phi$ is bijective, with inverse
	\[
	\Phi^{-1}(b_1,\ldots,b_m)=(b_m,b_1),
	\]
	which is regular. Hence $\Phi$ is an isomorphism between $\Omega$ and
	$X_\gamma$. Therefore $X_\gamma$ is smooth and irreducible and
	\[
	\dim_{\mathbb C}X_\gamma=2.
	\]
\end{proof}

We need one further observation. It shows that the first $m-2$ entries of a
cone-volume vector do not determine the last two entries.

\begin{lemma}
	\label{lem:cone_volume_vectors_with_different_last_entries}
	Let $U\in\U(2,m)$ and $\gamma\in C_{\cv}(U)$. Then there exists
	$\gamma'\in C_{\cv}(U)$ such that
	\[
	\gamma_i'=\gamma_i,\qquad i=1,\ldots,m-2, \text{ and } 
	\gamma_i'\neq\gamma_i,\qquad i=m-1,m.
	\]
\end{lemma}

For the proof we recall the subspace concentration condition in the discrete
setting. Let $U=(u_1,\ldots,u_m)\in\U(n,m)$ and
$\eta\in\R^m_{>0}$. We say that $\eta$ satisfies the
\emph{subspace concentration condition} with respect to $U$ if
\[
\sum_{u_i\in L}\eta_i
\leq \frac{\dim L}{n}\sum_{i=1}^m\eta_i
\]
for every linear subspace $L\subset\R^n$ with $0<\dim L<n$, and equality
can occur if and only if $L=\lin(L\cap U)$ and 
\[
\R^n
=\lin(L\cap U)
\oplus
\lin\bigl(U\setminus(L\cap U)\bigr).
\]
In particular, if $n=2$ and $U$ is irreducible, then the subspace
concentration condition is equivalent to the strict inequalities
\begin{equation}
	\sum_{u_i\in L}\eta_i
	<\frac12\sum_{i=1}^m\eta_i
	\label{eq:subspace_irr_dim2}
\end{equation}
for every one-dimensional subspace $L\subset\R^2$.

\begin{proof}[Proof of Lemma~\ref{lem:cone_volume_vectors_with_different_last_entries}]
	If $U$ is reducible, then $m=4$ and, with the cyclic ordering of the outer
	unit normals, we have 
	$u_3=-u_1$ and $u_4=-u_2$. 
	For every cone-volume vector of a parallelogram one has
	\[
	\gamma_1+\gamma_3=\gamma_2+\gamma_4.
	\]
	The vector
	\[
	\gamma'\coloneqq (\gamma_1,\gamma_2,\gamma_3+1,\gamma_4+1)
	\]
	also satisfies the subspace concentration condition with respect to $U$. By the sufficiency of this
	condition for the discrete logarithmic Minkowski problem, proved by Chen, Li and Zhu
	\cite{ChenLiZhu2019}, it belongs to $C_{\cv}(U)$ and has the required
	properties.
	
	Assume now that $U$ is irreducible. Set
	\[
	\tau\coloneqq \sum_{i=1}^{m-2}\gamma_i,
	\qquad
	L_{m-1}\coloneqq \lin\{u_{m-1}\},
	\qquad
	L_m\coloneqq \lin\{u_m\},
	\]
	and
	\[
	\nu_{m-1}
	\coloneqq \sum_{\substack{1\leq i\leq m-2\\u_i\in L_{m-1}}}\gamma_i,
	\qquad
	\nu_m
	\coloneqq \sum_{\substack{1\leq i\leq m-2\\u_i\in L_m}}\gamma_i.
	\]
	Since $u_{m-1}$ and $u_m$ are adjacent outer unit normals,
	$L_{m-1}\neq L_m$. Moreover,
	\[
	\nu_{m-1}+\nu_m<\tau.
	\]
	Indeed, equality would imply, since all $\gamma_i$ are positive, that
	$u_1,\ldots,u_{m-2}\in L_{m-1}\cup L_m$. Hence
	$U\subset L_{m-1}\cup L_m$, contradicting the irreducibility of $U$.
	
	For every one-dimensional subspace $L\subset\R^2$, set
	\[
	\nu_L
	\coloneqq \sum_{\substack{1\leq i\leq m-2\\u_i\in L}}\gamma_i.
	\]
	Choose $q$ sufficiently large such that
	\[
	\nu_L
	<q+\frac{\tau-\nu_{m-1}-\nu_m}{2}
	\]
	for every one-dimensional subspace
	$L\neq L_{m-1},L_m$ with $L\cap U\neq\emptyset$, and, in addition,
	\[
	q>\gamma_{m-1}+\nu_{m-1},
	\qquad
	q>\gamma_m+\nu_m.
	\]
	Define
	\[
	\gamma_i'\coloneqq \gamma_i,\qquad i=1,\ldots,m-2, 
	\text{ and } 
	\gamma_{m-1}'\coloneqq q-\nu_{m-1},
	\,
	\gamma_m'\coloneqq q-\nu_m.
	\]
	By the choice of $q$ we have 
	\[
	\gamma_{m-1}'>\gamma_{m-1}>0,
	\qquad
	\gamma_m'>\gamma_m>0.
	\]
	Furthermore,
	\[
	\frac12\sum_{i=1}^m\gamma_i'
	=q+\frac{\tau-\nu_{m-1}-\nu_m}{2}.
	\]
	On the two distinguished subspaces we have, by construction,
	\[
	\sum_{u_i\in L_{m-1}}\gamma_i'
	= 
	\sum_{u_i\in L_m}\gamma_i'
	=q < \frac12\sum_{i=1}^m\gamma_i'.
	\]
	Hence the strict inequality \eqref{eq:subspace_irr_dim2} holds on
	$L_{m-1}$ and $L_m$. For every other one-dimensional subspace
	$L$ with $L\cap U\neq\emptyset$, the masses on $L$ are unchanged, and thus
	\[
	\sum_{u_i\in L}\gamma_i'
	=\nu_L < \frac12\sum_{i=1}^m\gamma_i'.
	\]
	Therefore $\gamma'$ satisfies the subspace concentration condition. 
Again,  by the sufficiency of the
	subspace concentration condition \cite{ChenLiZhu2019}, we have 
	$\gamma'\in C_{\cv}(U)$, 
	which proves the assertion.
\end{proof}

We are now ready to prove Theorem~\ref{thm:planar_dim_leq_1}.

\begin{proof}[Proof of Theorem~\ref{thm:planar_dim_leq_1}]
	Set
	\[
	X_\gamma
	\coloneqq \mathcal V_{\mathbb C}
	(f_1-\gamma_1,\ldots,f_{m-2}-\gamma_{m-2})
	\]
	and
	\[
	Y_\gamma
	\coloneqq \mathcal V_{\mathbb C}
	(f_1-\gamma_1,\ldots,f_{m-1}-\gamma_{m-1}).
	\]
	By Lemma~\ref{lem:Variety_of_cone_volumes}, $X_\gamma$ is irreducible and
	$\dim_{\mathbb C}X_\gamma=2$. 
	Lemma~\ref{lem:cone_volume_vectors_with_different_last_entries} implies that
	$Y_\gamma$ is a proper algebraic subset of $X_\gamma$. Indeed, let
	$\gamma'\in C_{\cv}(U)$ be as in that lemma and choose
	$b'\in S(U,\gamma')$. Then
	$b'\in X_\gamma$, 
	whereas $\gamma'_{m-1}\neq\gamma_{m-1}$ gives
	$
	b'\notin Y_\gamma$. 
	Since every proper algebraic subset of an irreducible algebraic set has
	strictly smaller dimension,
	\[
	\dim_{\mathbb C}Y_\gamma\leq1.
	\]
	Finally, $
	S(U,\gamma)\subset Y_\gamma\cap\R^m$,  
	and using the standard comparison between the semialgebraic dimension
	of the real points and the complex dimension of a complex algebraic set,
	\[
	\dim^*S(U,\gamma)
	\leq\dim^*(Y_\gamma\cap\R^m)
	\leq\dim_{\mathbb C}Y_\gamma
	\leq1.
	\]
	
	The bound is sharp. Indeed, for a reducible $U\in\U(2,4)$ the irreducible
	decomposition has two components, and
	\cite[Proposition~5.1]{baumbach2025polynomialinequalitiesconevolumespolytopes}
	gives
	\[
	\dim^*S(U,\gamma)\geq1.
	\]
	Together with the upper bound just proved, this yields
	$\dim^*S(U,\gamma)=1$.
\end{proof}

\subsection{Lower bounds for non-uniqueness}
\label{subsec:planar-nonuniqueness}

The aim of this subsection is to show that, once at least four outer unit
normals are prescribed, uniqueness cannot be guaranteed by conditions on
$U$ alone.  The main ingredient is a stability argument showing that
non-uniqueness persists when additional outer unit normals are inserted.
We first collect the regularity properties of the cone-volume map that will
be used in this argument.  These properties are valid in arbitrary
dimension.

For $U\in\U(n,m)$ and $A\in\mathcal A(U)$, we recall that
$\det^A=\det Df_{\cv}^A$ denotes the Jacobian determinant of the
polynomial cone-volume map on $A$. Set
\[
W(U)\coloneqq 
\bigcup_{\substack{A\in\mathcal A(U), \det^A\not\equiv0}}
\bigl(\inter A\setminus\mathcal V(\det^A)\bigr).
\]

\begin{proposition}
	\label{prop:irreducible-regular-point}
	Let $U\in\U(n,m)$. Then
	$U$ is irreducible if and only if $W(U)\neq\emptyset$.
\end{proposition}

\begin{proof}
	Suppose first that $U$ is irreducible. By
	\cite[Proposition~2.7]
	{baumbach2025polynomialinequalitiesconevolumespolytopes},
	the normalized cone-volume set, i.e., $C_{\cv}(U)\cap \{y\in\R^m : y_1+\dots+y_m=1\}$ 
	contains a relatively open subset of dimension $m-1$.
	Hence, $C_{\cv}(U)$ contains an open subset of dimension $m$. 
	
	If $W(U)=\emptyset$, then $\det^A\equiv0$ for every
	full-dimensional type-cone $A\in\mathcal A(U)$. Hence every point of
	$\inter A$ is a critical point of the polynomial map $f_{\cv}^A$,
	and Sard's theorem (see, e.g., \cite{979f1be4-becf-36f2-be30-0bebf5263aa4}) 
	implies that $f_{\cv}^A(\inter A)$ has Lebesgue
	measure zero. The remaining parts of the type-cone decomposition
	have dimension at most $m-1$, and their images under the corresponding
	polynomial maps have also 
	Lebesgue measure zero. Since there are only finitely many type-cones,
	$C_{\cv}(U)$ would have Lebesgue measure zero, contradicting the fact
	that it has nonempty interior. Thus $W(U)\neq\emptyset$.
	
	Conversely, suppose that $W(U)\neq\emptyset$ and choose $b\in W(U)$.
	Then $b\in\inter A$ for some type-cone $A$ and
	$\det Df_{\cv}^A(b)\neq0$. By the inverse function theorem,
	$C_{\cv}(U)$ contains a nonempty open subset of $\R^m$ which again by 
	 \cite[Proposition~2.7]
	{baumbach2025polynomialinequalitiesconevolumespolytopes}
	shows that $U$ is irreducible.
\end{proof}

For an irreducible set $U\in\U(n,m)$ we set 
\[
B(U)\coloneqq 
f_{\cv}^U\bigl(\R^m_{>0}\setminus W(U)\bigr).
\]

\begin{proposition}
	\label{prop:regular-points}
	Let $U\in\U(n,m)$ be irreducible. Then the following holds: 
	\begin{enumerate}
		\item For every $b\in W(U)$, the map $f^U_{\cv}$ is differentiable
		and locally invertible at $b$. 
		\item If $\gamma\in C_{\cv}(U)\setminus B(U)$, then
		$S(U,\gamma)\subset W(U)$ and $|S(U,\gamma)|<\infty$.
				
		\item $W(U)$ is open and dense in
		\[
		\bigcup_{A\in\mathcal A(U),\, 
				\det^A\not\equiv0}\inter A.
		\]
		\item $B(U)$ is semialgebraic and has Lebesgue measure zero in
		$\R^m$.
	\end{enumerate}
\end{proposition}
\begin{proof}
	 i) follows directly from the inverse function
	theorem. For ii), 
	if $\gamma\notin B(U)$, then by the definition of $B(U)$ every
	$b\in S(U,\gamma)$ belongs to $W(U)$. Hence
	$\det Df_{\cv}^U(b)\neq0$, and the inverse function theorem yields a
	neighborhood $V_b$ of $b$ such that
	\[
	S(U,\gamma)\cap V_b=\{b\}.
	\]
	Thus every point of $S(U,\gamma)$ is isolated. Since
	$S(U,\gamma)$ is semialgebraic, it follows that $S(U,\gamma)$ is
	finite (see, e.g., \cite[Sections~2.3 and~2.8]{bochnak2013real}).
	
	For iii), if ${\det}^A\not\equiv0$, then its zero set has empty interior,
	so $\inter A\setminus\mathcal V({\det}^A)$ is open and dense in
	$\inter A$.  
	Finally, for iv) we note that $B(U)$ is semialgebraic by the
	Tarski--Seidenberg theorem as the complement of $W(U)$ is a finite union
	of type-cone boundaries, zero sets $\mathcal V({\det}^A)$ with
	${\det}^A\not\equiv0$, and type-cones on which ${\det}^A\equiv0$. Except for the latter all these sets have semi-algebraic dimension at most $m-1$ 
	and so do their images under $f^U_{\cv}$. If $A$ is a full-dimensional type cone on which ${\det}^A\equiv0$ then, as before, 
	Sard’s theorem shows that $f^A_{\mathrm{cv}}(\operatorname{int}A)$ has Lebesgue measure zero. Since this image is semialgebraic, it has semialgebraic dimension at most $m-1$.  
	Hence $\dim^*B(U)\leq m-1$ and it has Lebesgue measure zero; see, for example,
	\cite[Section~2.8]{bochnak2013real}. 	
\end{proof}
Before we continue, we present a short example. 
\begin{example}\label{exp:5.4}
	Let
	\begin{equation}
		U = \begin{pmatrix}
			0&\frac{\sqrt{2}}{2}&0&-1\\
			1&\frac{\sqrt{2}}{2}&-1&0
		\end{pmatrix} = (u_1,u_2,u_3,u_4),
	\end{equation}    
	and $A^*$ be the only proper type-cone, that is, $P(U,b)$ is a trapezoid for all $b\in\mathrm{int}(A^*)$. Then the derivative $Df_{\cv}(b)$ for $b\in\inter A^*$ is
	\[
	\begin{pmatrix}
		-b_1+\frac{\sqrt{2}}{2}b_2+\frac{1}{2}b_4 & \frac{\sqrt{2}}{2}b_1 & 0 & \frac{1}{2}b_1\\
		\frac{\sqrt{2}}{2}b_2 & \frac{\sqrt{2}}{2}(b_1+b_3) & \frac{\sqrt{2}}{2}b_2 & 0\\
		0 & \frac{\sqrt{2}}{2}b_3 & \frac{\sqrt{2}}{2}b_2+b_3+\frac{1}{2}b_4 & \frac{1}{2}b_3\\
		\frac{1}{2}b_4 & 0 & \frac{1}{2}b_4 & \frac{1}{2}(b_1+b_3)
	\end{pmatrix},
	\]
	and the determinant $\det^{A^*}$ is
	\[
	-\frac{1}{8}(b_1+b_3)^2\left(2b_2(b_1-b_3)+\sqrt{2}(2b_1b_3+b_1b_4-b_3b_4)\right),
	\]
	which is a nonzero polynomial.
	
	On the other hand, for 
	\begin{equation}
		U = \begin{pmatrix}
			0&1&0&-1\\
			1&0&-1&0
		\end{pmatrix} = (u_1,u_2,u_3,u_4),
	\end{equation}
	corresponding to a rectangle, the derivative $Df_{\cv}$ is
	\[
	\begin{pmatrix}
		\frac{1}{2}(b_2+b_4) & \frac{1}{2}b_1 & 0 & \frac{1}{2}b_1\\
		\frac{1}{2}b_2 & \frac{1}{2}(b_1+b_3) & \frac{1}{2}b_2 & 0\\
		0 & \frac{1}{2}b_3 & \frac{1}{2}(b_2+b_4) & \frac{1}{2}b_3\\
		\frac{1}{2}b_4 & 0 & \frac{1}{2}b_4 & \frac{1}{2}(b_1+b_3)
	\end{pmatrix}
	\]
	Denoting the rows by $a_i$, $i=1,\dots, 4$ we have $(a_1 + a_3) - (a_2 + a_4) = 0$ and so $\det^{A^*}\equiv 0$.  
\end{example}

Now we come back to the planar case. Then there is a unique proper type-cone, denoted by $A^*$.
If a type-cone is not proper, one of the prescribed facets is absent
throughout its interior, so one component of the corresponding polynomial
cone-volume map vanishes identically and consequently ${\det}^A\equiv0$.
Thus, for irreducible $U\in\U(2,m)$,
\[
W(U)=\inter A^*\setminus\mathcal V({\det}^{A^*}),
\qquad
B(U)=f_{\cv}^U\bigl(\R^m_{>0}\setminus W(U)\bigr).
\]

We next extract from the known quadrilateral results a regular
non-uniqueness example which will serve as the starting point for the
insertion argument.

\begin{theorem}
	\label{thm:4goncase}
	Let $U\in\U(2,4)$ be irreducible. Then there exists
	$\gamma\in C_{\cv}(U)\setminus B(U)$ such that
	$|S(U,\gamma)|\geq2$.
\end{theorem}

\begin{proof}
	An irreducible quadrilateral is either in general position, i.e., any two normals $u_i,u_j$, $i\ne j$,  
	are linearly independent, or a trapezoid.
	If $U$ is in general position, the classification of Liu et
	al.~\cite{LiuLuSunEtAl2024} implies that there exists a nonempty open set
	$\Omega\subset C_{\cv}(U)$ on which $S(U,\gamma)$ has at least two elements
	(in fact, one may choose $\Omega$ so that $|S(U,\gamma)|\geq 3$);
	see Proposition \ref{prop:liu-open-nonuniqueness} in  Appendix~\ref{app:liu-quadrilaterals}.  Since $B(U)$ has measure zero,
	we may choose $\gamma\in\Omega\setminus B(U)$.
	
	Suppose now that $U$ is the set of outer unit normals of a trapezoid.
	After a cyclic relabelling, let $u_1=-u_3$ and assume that
	$u_1,u_2,u_4$ lie in the same semicircle.  By
	\cite[Theorem~2.14]{Subspace_Concentration_of_Geometric_Measures}, one has
	$|S(U,\gamma)|=2$ whenever $\gamma=(\gamma_1,\gamma_2,\gamma_3,\gamma_4)\in\R^4_{>0}$ satisfies
	$\gamma_1<\gamma_3$ and
	$\gamma_1+\gamma_3\geq \gamma_2+\gamma_4>2\sqrt{\gamma_1\gamma_3}$.  The subset defined by the strict
	inequalities $\gamma_1<\gamma_3$ and
	$\gamma_1+\gamma_3>\gamma_2+\gamma_4>2\sqrt{\gamma_1\gamma_3}$ is nonempty and open.  It therefore
	contains a point outside $B(U)$, which proves the assertion.
\end{proof}

The following stability result is the main step for the proof of Theorem \ref{thm:1.7}. Its quantitative form
will also be used in the proof of Corollary~\ref{thm:1.8}. But first we recall the definition of the \emph{Hausdorff} distance: For two convex bodies $K,L\subset\R^n$ 
\begin{equation*}
	d_H\bigl(K,L)\coloneqq \max\{|h_K(u)-h_L(u)|:u\in\sph^{n-1}\}
\end{equation*} 
is called the Hausdorff distance of $K$ and $L$, and it induces a metric on the space of convex bodies in $\R^n$.

\begin{theorem}
	\label{thm:inserting-normal}
	Let $U=(u_1,\ldots,u_m)\in\U(2,m)$ be irreducible and let
	$\gamma\in C_{\cv}(U)\setminus B(U)$. Let
	$v\in\sph^1\setminus U$ and set
	\[
	U^+\coloneqq (u_1,\ldots,u_m,v).
	\]
	Let $S(U,\gamma)=\{b^{(1)},\ldots,b^{(N)}\}$ (cf. Prop.~\ref{prop:regular-points} ii))
	and, for $1\leq k\leq N$, set
	\[
	\widehat b^{(k)}
	\coloneqq 
	\bigl(b^{(k)},h_{P(U,b^{(k)})}(v)\bigr).
	\]
	Then, for every $\varepsilon,\delta>0$, there exist
	\[
	\gamma^+\in C_{\cv}(U^+)\setminus B(U^+)
	\]
	and pairwise distinct
	$b^{+,1},\ldots,b^{+,N}\in S(U^+,\gamma^+)$ such that
	\[
	\|\gamma^+-(\gamma,0)\|_\infty<\delta
	\]
	and, for $1\leq k\leq N$,
	\[
	\|b^{+,k}-\widehat b^{(k)}\|_\infty<\varepsilon,
	\qquad
	d_H\bigl(P(U^+,b^{+,k}),P(U,b^{(k)})\bigr)<\varepsilon.
	\]
	In particular,
	\[
	|S(U^+,\gamma^+)|
	\geq |S(U,\gamma)|.
	\]
\end{theorem}
\begin{proof}
	After a cyclic relabelling, we may assume that $v$ lies strictly
	between $u_m$ and $u_1$. By the linear invariance of the cone-volume
	problem
	\cite[Proposition~2.1]
	{baumbach2025polynomialinequalitiesconevolumespolytopes},
	we may further assume that $
	u_m=e_1$ and $u_1=e_2$. Observe that this linear coordinate transformation  changes  $\epsilon,\delta$ by a multiplicative constant and can be, therefore, neglected.   
	Write $
	v=\alpha u_m+\beta u_1$ with $\alpha,\beta>0$, $\alpha^2+\beta^2=1$. 
	Then $h_{P(U,b^{(k)})}(v)=\alpha b_m^{(k)}+\beta b_1^{(k)}$. Observe that $U^+$ is irreducible as well. 
	
	Let $A^*$ and $A_+^*$ be the proper type-cones associated with
	$U$ and $U^+$, respectively, and put $f\coloneqq f_{\cv}^{A^*}$, $F\coloneqq f_{\cv}^{A_+^*}$, and $\widetilde\gamma\coloneqq (\gamma,0)$. 
	
	For every $k$, the supporting line 
	$\{x\in\R^2 : \langle v,x\rangle = h_{P(U,b^{(k)})}(v)\}$ passes through the
  vertex determined by the edges with normals $u_m$ and $u_1$. Hence, 
	\[
	P(U^+,\widehat b^{(k)})=P(U,b^{(k)}),
	\qquad
	F(\widehat b^{(k)})=\widetilde\gamma.
	\]
	
	We compute the Jacobian of $F$ at $\widehat b^{(k)}$. For
	$b=(b_1,\ldots,b_m,b_{m+1})\in\R^{m+1}$, write
	$\bar b=(b_1,\ldots,b_m)$ and set
	\[
	q(b)\coloneqq b_{m+1}-\alpha b_m-\beta b_1.
	\]
	Using the planar cone-volume formula
	\eqref{eq:planar-edge-length}, and by some straightforward computations we find 
	\[
	F(b)=\binom{f(\bar b)}0+q(b)a(b),
	\]
	where
	\[
	a(b)=
	\left(
	\frac{b_1}{2\alpha},
	0,\ldots,0,
	\frac{b_m}{2\beta},
	-\frac{b_{m+1}}{2\alpha\beta}
	\right)^T.
	\]
	Since $q(\widehat b^{(k)})=0$, differentiation gives
	\[
	DF(\widehat b^{(k)})
	=
	\begin{pmatrix}
		Df(b^{(k)})&0\\
		0&0
	\end{pmatrix}
	+
	a(\widehat b^{(k)})
	\nabla q^T(\widehat b^{(k)}),
	\]
	with
	\[
	\nabla q(\widehat b^{(k)})=(-\beta,0,\ldots,0,-\alpha,1)^T.
	\]
	The last row is therefore
	$a_{m+1}(\widehat b^{(k)})\nabla q^T$. Subtracting
	$a_i(\widehat b^{(k)})/a_{m+1}(\widehat b^{(k)})$ times the last
	row from the $i$th row, $1\leq i\leq m$, yields a block triangular
	matrix and hence
	\begin{equation}
		\det DF(\widehat b^{(k)})
		=
		-\frac{\widehat b_{m+1}^{(k)}}{2\alpha\beta}
		\det Df_{\cv}^U(b^{(k)}).
		\label{eq:det-inserted-normal}
	\end{equation}
	By Proposition~\ref{prop:regular-points},
	$b^{(k)}\in W(U)$, and as $\widehat b_{m+1}^{(k)}>0$   the right-hand side is nonzero. The inverse function theorem now gives pairwise disjoint
	neighbourhoods $\mathcal N_k$ of $\widehat b^{(k)}$, a common
	neighbourhood $\mathcal O$ of $\widetilde\gamma$, and inverse branches
	\[
	\psi_k:\mathcal O\to\mathcal N_k,
	\qquad
	F(\psi_k(\eta))=\eta,
	\qquad
	\psi_k(\widetilde\gamma)=\widehat b^{(k)}.
	\]
	We also use the following elementary planar observation: if
	$b\in\R^{m+1}_{>0}$ and $F(b)\in\R^{m+1}_{>0}$, then
	$b\in\inter A_+^*$. Indeed, if the $i$th prescribed edge were absent,
	the signed edge length given by \eqref{eq:planar-edge-length} would be
	nonpositive, and hence $F_i(b)\leq0$, a contradiction.
	By shrinking $\mathcal O$ and the $\mathcal N_k$, we may assume that
	$\mathcal N_k\subset\R^{m+1}_{>0}$ and that, for every
	$\eta\in\mathcal O$,
	\[
	\|\eta-\widetilde\gamma\|_\infty<\delta,
	\qquad
	\|\psi_k(\eta)-\widehat b^{(k)}\|_\infty<\varepsilon,
	\]
	and
	\[
	d_H\bigl(P(U^+,\psi_k(\eta)),P(U,b^{(k)})\bigr)<\varepsilon
	\]
	for all $k$. Since $U^+$ is irreducible, $B(U^+)$ has Lebesgue measure zero by
	Proposition~\ref{prop:regular-points}. Hence we may choose
	\[
	\gamma^+
	\in\mathcal O\cap\R^{m+1}_{>0}\setminus B(U^+).
	\]
	Set $b^{+,k}\coloneqq \psi_k(\gamma^+)$. The preceding observation gives
	$b^{+,k}\in\inter A_+^*$, and so $f_{\cv}^{U^+}(b^{+,k})=\gamma^+$. 
	The $b^{+,k}$ are pairwise distinct because the $\mathcal N_k$ are
	disjoint, and all required estimates hold by construction. Thus
	\[
	|S(U^+,\gamma^+)|
	\geq N=|S(U,\gamma)|.
	\]
\end{proof}

Apparently, by applying Theorem~\ref{thm:inserting-normal} successively,
the same conclusion holds when finitely many new outer unit normals are
inserted. 
We can now prove the general planar non-uniqueness theorem.

\begin{proof}[Proof of Theorem~\ref{thm:1.7}]
	If $U$ is reducible, then $m=4$ and $U$ consists of two antipodal pairs $(u_1,u_2,-u_1,-u_2)$ and 
	\begin{equation*}
	\{\sqrt{|\det(u_1,u_2)|}(t,t^{-1},t, t^{-1}): t>0\} \subset S(U,(1,1,1,1)).
	\end{equation*} 	
	Assume that $U$ is irreducible.  If $m=4$, the result follows from
	Theorem~\ref{thm:4goncase}.  If $m\geq5$, choose an irreducible subtuple
	$U_0\subset U$ with $U_0\in\U(2,4)$.  By 
	Theorem~\ref{thm:4goncase} there is
	$\gamma^{(0)}\in C_{\cv}(U_0)\setminus B(U_0)$ with at least two
	solutions.  Finally, by successively applying Theorem \ref{thm:inserting-normal} we obtain a 
	cone-volume vector $\gamma$ with $|S(U,\gamma)|\geq 2$.
\end{proof}

The quantitative form of the insertion argument also yields the smooth
counterpart.

\begin{proof}[Proof of Corollary~\ref{thm:1.8}]
	Choose an irreducible $U_0\in\U(2,4)$ and
	$\gamma^{(0)}\in C_{\cv}(U_0)\setminus B(U_0)$ with two distinct
	solutions $b^{(0,1)}$ and $b^{(0,2)}$, as provided by
	Theorem~\ref{thm:4goncase}.  Let
	$K_j^{(0)}=P(U_0,b^{(0,j)})$, $j=1,2$, put
	$d_0=d_H(K_1^{(0)},K_2^{(0)})>0$ and let 
	\begin{equation*}
		r_0:=\min_{j=1,2}\min_{u\in S^1}h_{K_j^{(0)}}(u)>0,
		\qquad
		\rho:=\min\{d_0,r_0\}.
	\end{equation*}
	
	Let $D=\{v_1,v_2,\ldots\}\subset\mathbb S^1\setminus U_0$ be a countable
	dense set and define $U_k\coloneqq U_0\cup\{v_1,\ldots,v_k\}$.  Applying
	Theorem~\ref{thm:inserting-normal} inductively, we may choose
	$\gamma^{(k)}\in C_{\cv}(U_k)\setminus B(U_k)$ along with two different right hand sides 
	$b^{(k,1)},b^{(k,2)}\in S(U_k,\gamma^{(k)})$ such that, with
	$K_j^{(k)}=P(U_k,b^{(k,j)})$,
	\begin{equation}
		d_H(K_j^{(k+1)},K_j^{(k)})<4^{-(k+1)}\rho,
		\qquad
		\|\gamma^{(k+1)}-(\gamma^{(k)},0)\|_\infty
		<\frac13\min_{u\in U_k}\gamma_u^{(k)}.
		\label{eq:smooth-induction}
	\end{equation}
	Here $\gamma_u^{(k)}$ denotes the coordinate of $\gamma^{(k)}$
	corresponding to the vector $u\in U_k$.
	
	The first inequality in \eqref{eq:smooth-induction} shows that $K_j^{(k)}$ is a  Cauchy sequence in the Hausdorff metric.  Hence, let their limits
	be $K_1$ and $K_2$. For $j=1,2$, the triangle inequality gives
	\[
	d_H(K_j,K_j^{(0)})
	\leq \sum_{k=0}^{\infty}
	d_H(K_j^{(k+1)},K_j^{(k)})
	<\rho\sum_{k=0}^{\infty}4^{-(k+1)}
	=\frac{\rho}{3}.
	\]
	Consequently,
	\begin{equation*}
		d_H(K_1,K_2)
		\geq d_H(K_1^{(0)},K_2^{(0)})
		-d_H(K_1,K_1^{(0)})
		-d_H(K_2,K_2^{(0)})\ge d_0- \frac{2}{3}\rho\geq \frac{\rho}{3}>0,
	\end{equation*}
	and $h_{K_j}(u)\ge r_0-\frac{\rho}{3}>0$ for every $u\in\sph^1$. 
 Hence $K_1\neq K_2$ and $0\in\mathrm{int}(K_1), \mathrm{int}(K_2)$. 
	Since $V_{K_1^{(k)}}=V_{K_2^{(k)}}$ for every $k$, and the
	cone-volume measure is weakly continuous with respect to Hausdorff
	metric (cf.~\cite[Lemma 9.2.4]{Schneider2014}), we obtain
	\[
	V_{K_1}=V_{K_2}.
	\]
	It remains to verify smoothness.  Suppose that $K_1$ is not
	smooth, then there exists a non-regular $x$ in the boundary of $K_1$ whose normal cone $ N_{K_1}(x)$ (cf.~\cite[Section 2]{Schneider2014}) has dimension
	two.  By the density of $D$ we may choose an $r$ such that
	$v_r\in\relint N_{K_1}(x)$, and hence $V_{K_1}(\{v_r\})=0$.
	
	Set $\tau_k=\|\gamma^{(k+1)}-(\gamma^{(k)},0)\|_\infty$.  For $k\geq r$,
	the second inequality in \eqref{eq:smooth-induction} gives
	$\tau_k<\gamma_{v_r}^{(k)}/3$.  Moreover, the new cone-volume satisfies
	$\gamma_{v_{k+1}}^{(k+1)}\leq \tau_k$, and therefore
	$\tau_{k+1}<\tau_k/3$.  It follows that
	$\sum_{k=r}^{\infty}\tau_k<\gamma_{v_r}^{(r)}/2$.
	Moreover, for every $k\geq r$,
	\[
	\left|\gamma_{v_r}^{(k)}-\gamma_{v_r}^{(r)}\right|
	\leq \sum_{j=r}^{k-1} \tau_j
	< \frac12\gamma_{v_r}^{(r)}.
	\]
	Hence $
	\gamma_{v_r}^{(k)} >
	\frac12\gamma_{v_r}^{(r)}$ for all  $k\geq r$. 
	On the other hand, weak convergence of the cone-volume measures and the
	Portmanteau theorem give
	\[
	0=V_{K_1}(\{v_r\})
	\geq\limsup_{k\to\infty}V_{K_1^{(k)}}(\{v_r\})
	=\limsup_{k\to\infty}\gamma_{v_r}^{(k)}>0,
	\]
	a contradiction.  Thus $K_1$ is smooth, and the same argument applies to
	$K_2$.
\end{proof}

\section{Irreducible outer unit normals in $\mathbb R^n$}\label{sec:irreducible}

We begin with the proof of Proposition~\ref{thm:finiteness-dimension},
which relates finiteness of the fibers for the irreducible components of
$U$ to the dimension of the fibers for $U$ itself.

\begin{proof}[Proof of Proposition~\ref{thm:finiteness-dimension}]
	Let $U=S_1\cup\cdots\cup S_k$ 
	be the irreducible decomposition of $U$, and put
	$
	L_j\coloneqq \lin(S_j)$, 
	$n_j\coloneqq \dim L_j$. 
	By the linear invariance of the cone-volume problem
	\cite[Proposition~2.1]
	{baumbach2025polynomialinequalitiesconevolumespolytopes},
	we may assume that the subspaces $L_1,\ldots,L_k$ are mutually
	orthogonal.
	
	Fix $\gamma\in C_{\cv}(U)$. By homogeneity of the cone-volume map,
	we may assume that $\|\gamma\|_1=1$. For $j=1,\ldots,k$, let
	$\gamma_{S_j}\in\R^{|S_j|}$ denote the subvector of $\gamma$
	consisting of the coordinates corresponding to the normals in $S_j$,
	and set
	\[
	\delta^{(j)}\coloneqq \frac{n}{n_j}\gamma_{S_j}.
	\]
	By Proposition \ref{PropositionSeparatorsPsccAndConeVolumesDimension} we have 
	\[
	\delta^{(j)}\in C_{\cv}(S_j),
	\qquad
	\|\delta^{(j)}\|_1=1.
	\]
	Moreover we have 
	\[
	P(U,b)
	=
	P(S_1,b_{S_1})\oplus\cdots\oplus P(S_k,b_{S_k}),
	\]
	and with $
	v_j\coloneqq \vol_{n_j}\bigl(P(S_j,b_{S_j})\bigr)$, we may write  for 
	$b\in S(U,\gamma)$ (see also the proof of
	\cite[Proposition~2.7]{baumbach2025polynomialinequalitiesconevolumespolytopes})
	\begin{equation}
		\gamma_{S_j}
		=
		\frac{n_j}{n}\,
		\frac{\gamma(S_j,b_{S_j})}{v_j},
		\qquad j=1,\ldots,k.
		\label{eq:normalized-direct-sum}
	\end{equation}
	Observe that $\prod_jv_j=\vol_n(P(U,b))=1$. Let 
	\[
	T\coloneqq 
	\left\{
	t=(t_1,\ldots,t_k)\in\R^k_{>0}:
	\prod_{j=1}^k t_j^{n_j}=1
	\right\}
	\]
	and we define 
	\[
	\Phi:
	T\times\prod_{j=1}^k S(S_j,\delta^{(j)})
	\longrightarrow S(U,\gamma) \text{ with }
	\Phi(t,c^{(1)},\ldots,c^{(k)})_{S_j}
	\coloneqq t_jc^{(j)}.
	\]
	Since $\|\delta^{(j)}\|_1=1$, every
	$c^{(j)}\in S(S_j,\delta^{(j)})$ determines a polytope of
	$n_j$-dimensional volume one. Scaling it by $t_j$ multiplies both its
	volume and its cone-volume vector by $t_j^{n_j}$. Hence its normalized
	cone-volume vector remains $\delta^{(j)}$. Formula
	\eqref{eq:normalized-direct-sum} therefore gives
	\[
	\gamma\bigl(U,\Phi(t,c^{(1)},\ldots,c^{(k)})\bigr)_{S_j}
	=
	\frac{n_j}{n}\delta^{(j)}
	=
	\gamma_{S_j}.
	\]
	Thus $\Phi$ is well defined.
	
	Conversely, let $b\in S(U,\gamma)$ and set 
	$t_j\coloneqq v_j^{1/n_j}$,
	and 
	$c^{(j)}\coloneqq t_j^{-1}b_{S_j}$. 
	Since $\vol_n(P(U,b))=1$, we have
	\[
	\prod_{j=1}^k t_j^{n_j}
	=
	\prod_{j=1}^k v_j
	=1,
	\]
	so $t\in T$. Moreover, by
	\eqref{eq:normalized-direct-sum},
	\[
	\gamma(S_j,c^{(j)})
	=
	\frac{1}{v_j}\gamma(S_j,b_{S_j})
	=
	\frac{n}{n_j}\gamma_{S_j}
	=
	\delta^{(j)}.
	\]
	Hence $c^{(j)}\in S(S_j,\delta^{(j)})$. Thus $\Phi$ is a
	semialgebraic bijection with semialgebraic inverse.
	
	Since $\dim^*T=k-1$, additivity of semialgebraic dimension under
	Cartesian products yields
	\begin{equation}
		\dim^*S(U,\gamma)
		=
		(k-1)
		+\sum_{j=1}^k
		\dim^*S(S_j,\delta^{(j)}).
		\label{eq:fiber-dimension-decomposition}
	\end{equation}
	
	Assume now first that all fibers associated with the irreducible components
	$S_j$ are finite. Then $\dim^*S(S_j,\delta^{(j)})=0$  hence $
	\dim^*S(U,\gamma)=k-1$. 
	
	Conversely, suppose that
	\[
	\dim^*S(U,\gamma)=k-1
	\]
	for every $\gamma\in C_{\cv}(U)$.  Fix $i\in\{1,\ldots,k\}$ and
	$\eta\in C_{\cv}(S_i)$. By homogeneity, it suffices to consider
	$\|\eta\|_1=1$. For every $j\neq i$, choose
	$\delta^{(j)}\in C_{\cv}(S_j)$ with
	$\|\delta^{(j)}\|_1=1$, and set $\delta^{(i)}\coloneqq \eta$.
	As before we conclude 
	$
	\gamma_{S_j}=\frac{n_j}{n}\delta^{(j)}$ and from 
	 \eqref{eq:fiber-dimension-decomposition} we get 
	$
	\dim^*S(S_i,\eta)=0$, and thus  $S(S_i,\eta)$ is finite. 
\end{proof}

We next prove Theorem~\ref{thm:generic-finiteness-intro}. In fact, the
proof gives the slightly stronger estimate
\[
|S(U,\gamma)|
\leq
\left(
\sum_{i=0}^{m}
\binom{\binom{m}{n+1}}{i}
\right)
n(2n-1)^{m-1}.
\]

\begin{proof}[Proof of Theorem~\ref{thm:generic-finiteness-intro}]
	By Proposition~\ref{prop:regular-points}(ii), every
	$b\in S(U,\gamma)$ belongs to $W(U)$. Hence $b\in\inter A$ for some
	proper type-cone $A$ and
	\[
	\det Df_{\cv}^A(b)\neq0.
	\]
	In particular, $b$ is an isolated point, and hence a connected component,
	of the real algebraic set $\mathcal V(f_1^A-\gamma_1,\ldots,f_m^A-\gamma_m)$. 
	The polynomials $f_i^A-\gamma_i$ have degree at most $n$. By the
	Thom--Milnor bound \cite[Theorem~2]{Milnor1964}, the number of connected
	components of this algebraic set is at most
	$n(2n-1)^{m-1}$. Therefore
	\[
	|S(U,\gamma)\cap\inter A|
	\leq n(2n-1)^{m-1}.
	\]
	
	It remains to bound the number of full-dimensional type-cones. Their
	walls are contained in the hyperplanes determined by affine dependencies
	among $n+1$ of the supporting hyperplanes. There are at most
	$\binom{m}{n+1}$ such hyperplanes. Hence, by the standard bound for the
	number of regions of a hyperplane arrangement
	\cite[\S16]{schlafli1901theorie},
	\[
	|\mathcal A(U)|
	\leq
	\sum_{i=0}^{m}
	\binom{\binom{m}{n+1}}{i}.
	\]
	Summing the preceding estimate over the type-cones gives the first bound;
	the final inequality is an elementary estimate.
\end{proof}

We finish the section by showing that, in general position, the number of
solutions is locally constant away from the exceptional set $B(U)$.

\begin{theorem}
	\label{thm:con} 
	Let $U\in\U(n,m)$ be in general position. Then
	$|S(U,\gamma)|$ is constant on every path-connected component of
	$C_{\cv}(U)\setminus B(U)$.
\end{theorem}

\begin{proof}
	Observe, since $U$ positively spans $\mathbb R^n$, general position implies that
	$U$ is irreducible: a proper direct-sum decomposition would force one of
	its summands to contain more vectors than its dimension, contradicting
	general position. Thus $B(U)$ is defined. Zhu's \cite{Zhu2014} existence
	theorem for discrete data in general position  gives (cf.~\cite[Theorem~III]{baumbach2025polynomialinequalitiesconevolumespolytopes})
	\[
	C_{\cv}(U)=\mathbb R^m_{>0}.
	\]
	Set
	\[
	X\coloneqq \{b\in\mathbb R^m_{>0}:f_{\cv}^U(b)\in\mathbb R^m_{>0}\}.
	\]
	We first claim that
	\[
	f_{\cv}^U:X\longrightarrow\mathbb R^m_{>0}
	\]
	is proper that is, the preimage of every compact subset of
	$\R^m_{>0}$ is compact.

	 Let $M\subset\mathbb R^m_{>0}$ be compact and
	$b\in(f_{\cv}^U)^{-1}(M)$. Since the total mass of the cone-volume
	measure equals the volume,
	\[
	\vol_n(P(U,b))=\|f_{\cv}^U(b)\|_1.
	\]
	Thus these volumes are bounded above and away from zero on
	$(f_{\cv}^U)^{-1}(M)$. After rescaling each polytope to volume one,
	Zhu's compactness theorem \cite[Theorem~4.3]{Zhu2014} gives a uniform
	bound on their diameters. Since every coordinate of
	$f_{\cv}^U(b)$ is positive, every $u_i$ determines a facet and
	$b_i=h_{P(U,b)}(u_i)$. Hence, $(b_1,\dots,b_m)$ is 
	bounded as well, and thus $(f_{\cv}^U)^{-1}(M)$ is bounded. If a sequence in
	this preimage converges to $b\in\mathbb R^m_{\geq0}$, continuity of the
	cone-volume map gives $f_{\cv}^U(b)\in M\subset\mathbb R^m_{>0}$.
	This forces $b\in X$, and therefore the preimage is closed in
	$\mathbb R^m$. Hence it is compact, proving properness.
	
	Since $W(U)$ is open in $X$, the set $X\setminus W(U)$ is closed in $X$.
	Properness therefore implies that
	\[
	B(U)\cap\mathbb R^m_{>0}
	=f_{\cv}^U(X\setminus W(U))
	\]
	is closed in $\mathbb R^m_{>0}$. Consequently,
	\[
	Y\coloneqq C_{\cv}(U)\setminus B(U)
	\]
	is open. Put $D\coloneqq (f_{\cv}^U)^{-1}(Y)$. By the definition of $B(U)$,
	$D\subset W(U)$, and Proposition~\ref{prop:regular-points}(i) shows that
	\[
	f_{\cv}^U|_D:D\longrightarrow Y
	\]
	is a local diffeomorphism. It is also proper, hence it is a finite covering map; see
	\cite[Lemma~3]{ho1975note}. 
	
	Hence the restriction of $f_{\cv}^U$ is a finite covering map.
	By the path-lifting property for covering maps
	\cite[Proposition~1.30]{Hatcher2002}, any path in
	$C_{\cv}(U)\setminus B(U)$ joining $\gamma$ to $\gamma'$ induces a
	bijection between $S(U,\gamma)$ and $S(U,\gamma')$.
	Thus
	\[
	|S(U,\gamma)|=|S(U,\gamma')|.
	\]
\end{proof}

By Theorem~\ref{thm:con}, for $U$ in general position the cardinality of
$S(U,\gamma)$ is constant on every path-connected component of
$C_{\cv}(U)\setminus B(U)$, but it may change when $\gamma$ crosses
$B(U)$. The quadrilateral classification of Liu et al.~(see
Table~\ref{tab:liu-quadrilateral}), Example~\ref{exp:pentagon} and
numerical experiments indicate that it can only change by an even number.
Theorem~\ref{thm:parity} says this is indeed the case. In fact, as we will show next, the following stronger version holds:

\begin{theorem}
	\label{thm:parity2}
	Let $U\in\U(n,m)$ be in general position and 
	$\gamma\in C_{\cv}(U)\setminus B(U)$. Then
	\[
	\sum_{b\in S(U,\gamma)}\sgn\det Df_{\cv}^U(b)=(-1)^{m+n-1}.
	\]
	In particular, $|S(U,\gamma)|$ is odd.
\end{theorem}

The proof is based on a variational characterization of the solutions,
which goes back to \cite{boroczky2013logarithmic_JAMS,Zhu2014}, combined
with an index count. It occupies the rest of this section.

\medskip
\noindent\emph{Notation.} Let
$E=\{h\in\R^m:\inter P(U,h)\neq\emptyset\}$ and
$v(h)=\vol(P(U,h))$ for $h\in E$. We write
$\Lambda=\{U^\intercal t:t\in\R^n\}$; as $U$ spans $\R^n$, the map
$t\mapsto U^\intercal t$ is injective. Since
$P(U,h-U^\intercal t)=P(U,h)-t$, we have $E=\R^m_{>0}+\Lambda$ and
\begin{equation}
	\label{eq:parity-transl}
	v(h-U^\intercal t)=v(h),\qquad h\in E,\ t\in\R^n.
\end{equation}
It is classical that $v$ is continuously differentiable on $E$ with
\begin{equation}
	\label{eq:parity-dv}
	\frac{\partial v}{\partial h_i}(h)=\vol_{n-1}(F_i(h)),
	\qquad 1\leq i\leq m,
\end{equation}
see \cite[Section~5.1]{Schneider2014}: on every type-cone this follows
from Lemma~\ref{res:type_cones_and_volume_polynomials} and
\eqref{eq:facetineq}, and the right hand sides of \eqref{eq:parity-dv}
are continuous. Consequently,
\begin{equation}
	\label{eq:parity-fcvdv}
	f_{\cv}^U(b)_i=\frac{b_i}{n}\frac{\partial v}{\partial b_i}(b),
	\qquad b\in\R^m_{>0}.
\end{equation}
Both $v$ and $f_{\cv}^U$ are positively homogeneous of degree $n$. Hence
$S(U,\tau^n\gamma)=\tau S(U,\gamma)$ for $\tau>0$, the sets $W(U)$ and
$B(U)$ are invariant under such scalings, and
$\det Df_{\cv}^U(\tau b)=\tau^{m(n-1)}\det Df_{\cv}^U(b)$. Therefore it
suffices to prove Theorem~\ref{thm:parity2} for $\gamma$ with
$\|\gamma\|_1=1$, which we assume from now on.

For $b\in\R^m_{>0}$ we define
\[
\Psi_\gamma(b)=\sum_{i=1}^m\gamma_i\log b_i-\frac1n\log v(b),
\]
and we set
\[
M_\gamma=\Bigl\{b\in\R^m_{>0}:
\sum_{i=1}^m\frac{\gamma_i}{b_i}u_i=0,\ v(b)=1\Bigr\}.
\]

\begin{lemma}
	\label{lem:parity-functional}
	\begin{enumerate}
		\item $\Psi_\gamma\in C^1(\R^m_{>0})$ and
		$\Psi_\gamma(\tau b)=\Psi_\gamma(b)$ for $\tau>0$.

		\item $\nabla\Psi_\gamma(b)=0$ if and only if
		$f_{\cv}^U(b)=v(b)\gamma$. In particular,
		$S(U,\gamma)=\{b\in\R^m_{>0}:\nabla\Psi_\gamma(b)=0,\ v(b)=1\}$.

		\item For $h\in E$ the function
		$t\mapsto\Psi_\gamma(h-U^\intercal t)$ is defined and smooth on
		$\inter P(U,h)$, strictly concave, and tends to $-\infty$ at the
		boundary of $P(U,h)$. Hence it has a unique maximum point
		$t_\gamma(h)\in\inter P(U,h)$, and $t_\gamma$ is smooth on $E$.
		Moreover, $t_\gamma(h)=0$ if and only if
		$\sum_i\gamma_iu_i/h_i=0$.

		\item $S(U,\gamma)\subseteq M_\gamma$.
	\end{enumerate}
\end{lemma}

\begin{proof}
	i) follows from \eqref{eq:parity-dv} and the homogeneity of $v$, and
	ii) from
	\[
	\frac{\partial\Psi_\gamma}{\partial b_i}(b)
	=\frac{\gamma_i}{b_i}-\frac{1}{nv(b)}\frac{\partial v}{\partial b_i}(b)
	=\frac{\gamma_iv(b)-f_{\cv}^U(b)_i}{b_i\,v(b)},
	\]
	together with $\sum_if_{\cv}^U(b)_i=v(b)$ and $\|\gamma\|_1=1$.

	iii) By \eqref{eq:parity-transl},
	$\Psi_\gamma(h-U^\intercal t)
	=\sum_i\gamma_i\log(h_i-\langle u_i,t\rangle)-\frac1n\log v(h)$, and
	$h-U^\intercal t>0$ if and only if $t\in\inter P(U,h)$. The Hessian
	with respect to $t$ is
	$-\sum_i\gamma_i(h_i-\langle u_i,t\rangle)^{-2}u_iu_i^\intercal$,
	which is negative definite as $U$ spans $\R^n$. Since $P(U,h)$ is
	bounded, the function tends to $-\infty$ at the boundary. The
	smoothness of $t_\gamma$ follows from the implicit function theorem,
	and the last assertion from the fact that the gradient with respect to
	$t$ at $t=0$ equals $-\sum_i\gamma_iu_i/h_i$.

	iv) For $b\in S(U,\gamma)$ we have $v(b)=\|\gamma\|_1=1$ and
	$\gamma_i/b_i=\frac1n\vol_{n-1}(F_i(b))$. Hence the assertion follows
	from Minkowski's relation $\sum_i\vol_{n-1}(F_i(b))u_i=0$.
\end{proof}

Next we eliminate the translations and the scaling. Let
$\pi:\R^m\to\R^{m-n}$ be a linear map with kernel $\Lambda$. Since
$\pos(U)=\R^n$, there is $\omega\in\R^m_{>0}$ with
$\sum_i\omega_iu_i=0$, i.e., $\omega\perp\Lambda$, and hence there is a
linear form $\ell$ on $\R^{m-n}$ with
$\ell\circ\pi=\langle\omega,\cdot\rangle$. The set
$C=\pi(\R^m_{>0})=\pi(E)$ is an open convex cone, and
\[
\mathcal Y=\{y\in C:\ell(y)=1\}
\]
is a non-empty open convex subset of the affine hyperplane $\{\ell=1\}$,
whose dimension is
\[
d=m-n-1.
\]
For $y\in\mathcal Y$ choose $h\in E$ with $\pi(h)=y$, let
$\rho_\gamma(h)=h-U^\intercal t_\gamma(h)\in\R^m_{>0}$ and put
\[
\Theta_\gamma(y)=v(h)^{-1/n}\rho_\gamma(h),\qquad
\psi_\gamma(y)=\Psi_\gamma(\Theta_\gamma(y))
=\sum_{i=1}^m\gamma_i\log\Theta_\gamma(y)_i.
\]

\begin{lemma}
	\label{lem:parity-slice}
	\begin{enumerate}
		\item $\Theta_\gamma:\mathcal Y\to M_\gamma$ is well defined and a
		homeomorphism, with inverse
		$b\mapsto\pi(b)/\langle\omega,b\rangle$. Moreover, $\Theta_\gamma$
		and $\psi_\gamma$ are of class $C^1$, and they are smooth in a
		neighbourhood of every $y$ for which $\Theta_\gamma(y)$ lies in the
		interior of a type-cone.

		\item $y\in\mathcal Y$ is a critical point of $\psi_\gamma$ if and
		only if $\Theta_\gamma(y)\in S(U,\gamma)$. In particular,
		$\Theta_\gamma$ restricts to a bijection between the critical
		points of $\psi_\gamma$ and $S(U,\gamma)$.
	\end{enumerate}
\end{lemma}

\begin{proof}
	i) If $\pi(h')=\pi(h)$, then $h'=h+U^\intercal s$ for some
	$s\in\R^n$, and $P(U,h')=P(U,h)+s$. Hence
	$t_\gamma(h')=t_\gamma(h)+s$, $\rho_\gamma(h')=\rho_\gamma(h)$ and
	$v(h')=v(h)$, so $\Theta_\gamma$ is well defined. By
	Lemma~\ref{lem:parity-functional}(iii),
	$\sum_i\gamma_iu_i/\rho_\gamma(h)_i=0$; this condition is invariant
	under scaling, and by \eqref{eq:parity-transl}
	$v(\Theta_\gamma(y))=1$. Thus $\Theta_\gamma(y)\in M_\gamma$.

	For $b\in M_\gamma$ let $y=\pi(b)/\langle\omega,b\rangle$. Then $y\in
	C$ and $\ell(y)=1$, i.e., $y\in\mathcal Y$. With
	$h=b/\langle\omega,b\rangle$ we have $\pi(h)=y$ and $t_\gamma(h)=0$ by
	Lemma~\ref{lem:parity-functional}(iii), so $\rho_\gamma(h)=h$,
	$v(h)=\langle\omega,b\rangle^{-n}$ and $\Theta_\gamma(y)=b$.
	Conversely, for $y\in\mathcal Y$ and $\pi(h)=y$, we have
	$\pi(\Theta_\gamma(y))=v(h)^{-1/n}y$, and applying $\ell$ gives
	\[
	\pi(\Theta_\gamma(y))/\langle\omega,\Theta_\gamma(y)\rangle=y.
	\]
	Hence the two maps are inverse to each other; both are continuous.

	As $\pi|_{\Lambda^\perp}$ is an isomorphism and $E=\pi^{-1}(C)$, the
	map $\sigma(y)=(\pi|_{\Lambda^\perp})^{-1}(y)$ is a smooth section
	$\mathcal Y\to E$. Since
	$\Theta_\gamma=v(\sigma)^{-1/n}\,\rho_\gamma\circ\sigma$, the
	regularity assertions follow from
	Lemma~\ref{lem:parity-functional}(iii), \eqref{eq:parity-dv},
	\eqref{eq:parity-transl} and the fact that $v$ is a polynomial on
	every type-cone.

	ii) Let
	$J_\gamma(h)=\Psi_\gamma(\rho_\gamma(h))
	=\max_t\Psi_\gamma(h-U^\intercal t)$, $h\in E$. As the derivative with
	respect to $t$ vanishes at $t=t_\gamma(h)$, the chain rule gives
	\begin{equation}
		\label{eq:parity-envelope}
		\nabla J_\gamma(h)=\nabla\Psi_\gamma(\rho_\gamma(h)).
	\end{equation}
	Moreover, $J_\gamma$ is invariant under translations $h\mapsto
	h+U^\intercal s$ and scalings $h\mapsto\tau h$, $\tau>0$. Therefore
	$J_\gamma=\psi_\gamma\circ q$ with
	$q(h)=\pi(h)/\langle\omega,h\rangle$. Since
	$\langle\omega,h\rangle>0$ on $E$ and $q$ coincides with the affine
	surjection $\pi$ on $E\cap\{\langle\omega,h\rangle=1\}$, the map
	$q:E\to\mathcal Y$ is a smooth submersion. Consequently, $y=q(h)$ is a
	critical point of $\psi_\gamma$ if and only if $\nabla J_\gamma(h)=0$,
	which by \eqref{eq:parity-envelope} and
	Lemma~\ref{lem:parity-functional}(ii) is equivalent to
	$f_{\cv}^U(\rho_\gamma(h))=v(\rho_\gamma(h))\gamma$. By homogeneity,
	this holds if and only if
	$\Theta_\gamma(y)=v(\rho_\gamma(h))^{-1/n}\rho_\gamma(h)\in
	S(U,\gamma)$. The last statement follows from i) and
	Lemma~\ref{lem:parity-functional}(iv).
\end{proof}

The hypothesis that $U$ is in general position enters only through the
following lemma, via Zhu's compactness theorem.

\begin{lemma}
	\label{lem:parity-coercive}
	Let $U\in\U(n,m)$ be in general position. Then there is $\delta>0$
	such that $b_i\geq\delta$ for all $b\in M_\gamma$ and $1\leq i\leq m$.
	Consequently, $\psi_\gamma$ is bounded from below by $\log\delta$, and
	all its sublevel sets $\{y\in\mathcal Y:\psi_\gamma(y)\leq c\}$ are
	compact.
\end{lemma}

\begin{proof}
	Suppose that $\min_ib^{(k)}_i\to0$ for a sequence of points
	$b^{(k)}\in M_\gamma$. Put $\mu_k=\max_i(\gamma_i/b^{(k)}_i)$, so that
	$\mu_k\to\infty$, and let
	$c^{(k)}_i=\gamma_i/(\mu_kb^{(k)}_i)\in(0,1]$. Then
	$\max_ic^{(k)}_i=1$ and $\sum_ic^{(k)}_iu_i=0$. After passing to a
	subsequence, $c^{(k)}\to c\in[0,1]^m\setminus\{0\}$ with
	$\sum_ic_iu_i=0$, and $b^{(k)}_i=\gamma_i/(\mu_kc^{(k)}_i)\to0$ for
	all $i\in Z=\{i:c_i>0\}$.

	The polytopes $P_k=P(U,b^{(k)})$ contain the origin and have volume
	one. Hence, by Zhu's compactness theorem
	\cite[Theorem~4.3]{Zhu2014}, they are uniformly bounded; note that
	general position in \cite{Zhu2014} includes that the vectors are not
	contained in a closed hemisphere, which holds since $\pos(U)=\R^n$. By
	Blaschke's selection theorem we may assume that $P_k\to P_\infty$ in
	the Hausdorff metric, where $\vol(P_\infty)=1$. Every $x\in P_\infty$
	is the limit of points $x_k\in P_k$, and so
	$\langle u_i,x\rangle\leq\lim_kb^{(k)}_i=0$ for $i\in Z$. As
	$\sum_{i\in Z}c_i\langle u_i,x\rangle=0$ with $c_i>0$, all these inner
	products vanish, i.e., $P_\infty\subseteq u_i^\perp$ for $i\in Z$.
	This contradicts $\vol(P_\infty)=1$.

	On $M_\gamma$ we have
	$\Psi_\gamma(b)=\sum_i\gamma_i\log b_i\geq\log\delta$. If
	$\Psi_\gamma(b)\leq c$, then
	$\gamma_i\log b_i\leq c-(1-\gamma_i)\log\delta$, so $b$ lies in a
	compact box contained in $\R^m_{>0}$. As the conditions defining
	$M_\gamma$ are closed in $\R^m_{>0}$, the set
	$\{b\in M_\gamma:\Psi_\gamma(b)\leq c\}$ is compact. Since
	$\psi_\gamma=\Psi_\gamma\circ\Theta_\gamma$ and $\Theta_\gamma$ is a
	homeomorphism, the assertions follow.
\end{proof}

In the next lemma the Hessian of $\psi_\gamma$ is taken with respect to
affine coordinates of the hyperplane $\{\ell=1\}$. At a critical point
the sign of its determinant does not depend on this choice. For a
non-degenerate symmetric bilinear form $\mathfrak q$ on a real vector
space we write $\varepsilon(\mathfrak q)\in\{\pm1\}$ for the sign of the
determinant of a Gram matrix of $\mathfrak q$, i.e.,
$\varepsilon(\mathfrak q)=(-1)^k$, where $k$ is the number of negative
eigenvalues; $\varepsilon$ is multiplicative with respect to orthogonal
direct sums. For $d=0$ we use the convention that the empty determinant
equals $1$.

\begin{lemma}
	\label{lem:parity-index}
	Let $U\in\U(n,m)$ be in general position, let
	$\gamma\in\R^m_{>0}\setminus B(U)$, $b\in S(U,\gamma)$ and
	$y_0=\Theta_\gamma^{-1}(b)$. Then $\psi_\gamma$ is smooth in a
	neighbourhood of $y_0$, $\mathrm{Hess}\,\psi_\gamma(y_0)$ is
	non-singular, and
	\[
	\sgn\det\mathrm{Hess}\,\psi_\gamma(y_0)
	=(-1)^{m+n-1}\sgn\det Df_{\cv}^U(b).
	\]
\end{lemma}

\begin{proof}
	By Proposition~\ref{prop:regular-points}(ii), $b\in W(U)$, i.e.,
	$b\in\inter A$ for a type-cone $A$ with $\det^A(b)\neq0$. On
	$\inter A$ we have $v=v^A$ and $f_{\cv}^U=f_{\cv}^A$, which are
	polynomials, homogeneous of degree $n$. Hence $\Psi_\gamma$ is smooth
	near $b$, and by Lemma~\ref{lem:parity-slice}(i), $\psi_\gamma$ is
	smooth near $y_0$. Since $b$ is a critical point of $\Psi_\gamma$
	(Lemma~\ref{lem:parity-functional}(ii)), the second differential
	$\mathfrak h$ of $\Psi_\gamma$ at $b$ is a well-defined symmetric
	bilinear form on $\R^m$, and by the invariance of $\Psi_\gamma$ under
	scalings, $b$ lies in its radical. We compute $\mathfrak h$ in two
	coordinate systems.

	\emph{Logarithmic coordinates.} Let $x_0\in\R^m$ be given by
	$(x_0)_i=\log b_i$, let $V(x)=v(e^{x_1},\dots,e^{x_m})$, and consider
	\[
	\widehat\Psi(x)=\Psi_\gamma(e^{x_1},\dots,e^{x_m})
	=\langle\gamma,x\rangle-\tfrac1n\log V(x).
	\]
	By \eqref{eq:parity-fcvdv}, $\nabla V=n\,f_{\cv}^A\circ\exp$, and
	hence $\nabla\widehat\Psi=\gamma-(f_{\cv}^A\circ\exp)/V$. Using
	$V(x_0)=1$ and $f_{\cv}^A(b)=\gamma$ we obtain
	\[
	\mathrm{Hess}\,\widehat\Psi(x_0)
	=-\bigl(\mathcal A-n\gamma\gamma^\intercal\bigr),
	\qquad
	\mathcal A=Df_{\cv}^A(b)\,\mathrm{diag}(b_1,\dots,b_m).
	\]
	Here $\mathcal A=\frac1n\mathrm{Hess}\,V(x_0)$ is symmetric, and
	Euler's relation for the homogeneous map $f_{\cv}^A$ gives
	$\mathcal A\mathbf 1=n\gamma$, where
	$\mathbf 1=(1,\dots,1)^\intercal$. Consider the decomposition
	$\R^m=\R\mathbf 1\oplus\gamma^\perp$ (note
	$\langle\gamma,\mathbf 1\rangle=1$). For $w\in\gamma^\perp$ we have
	$\mathbf 1^\intercal\mathcal Aw=n\langle\gamma,w\rangle=0$ and
	$\mathrm{Hess}\,\widehat\Psi(x_0)w=-\mathcal Aw$. Hence, with respect
	to this decomposition,
	\[
	\mathcal A\cong(n)\oplus\mathcal A|_{\gamma^\perp},
	\qquad
	\mathrm{Hess}\,\widehat\Psi(x_0)
	\cong(0)\oplus\bigl(-\mathcal A|_{\gamma^\perp}\bigr).
	\]
	As $\det\mathcal A=\det Df_{\cv}^U(b)\,b_1\cdots b_m\neq0$, the form
	$\mathcal A|_{\gamma^\perp}$ is non-degenerate with
	$\varepsilon(\mathcal A|_{\gamma^\perp})=\sgn\det Df_{\cv}^U(b)$.
	Since the differential of the exponential map at $x_0$ maps
	$\mathbf 1$ to $b$, we conclude that the radical of $\mathfrak h$ is
	$\R b$ and that the induced form $\bar{\mathfrak h}$ on
	$\R^m/\R b$ satisfies
	\begin{equation}
		\label{eq:parity-eps}
		\varepsilon(\bar{\mathfrak h})
		=(-1)^{m-1}\sgn\det Df_{\cv}^U(b).
	\end{equation}

	\emph{Coordinates adapted to translations and scalings.} For
	$(y,t,s)$ in a neighbourhood of $(y_0,0,0)$ in
	$\mathcal Y\times\R^n\times\R$ let
	$R(y,t,s)=e^s\bigl(\Theta_\gamma(y)-U^\intercal t\bigr)$. The map $R$
	is smooth, $R(y_0,0,0)=b$, and it has the smooth local inverse
	\[
	b'\mapsto
	\Bigl(q(b'),\
	(UU^\intercal)^{-1}U\bigl(\Theta_\gamma(q(b'))-e^{-s}b'\bigr),\ s\Bigr),
	\qquad
	s=\log\frac{\langle\omega,b'\rangle}
	{\langle\omega,\Theta_\gamma(q(b'))\rangle},
	\]
	with $q$ as in the proof of Lemma~\ref{lem:parity-slice}; this follows
	from $q\circ\Theta_\gamma=\id$ and
	$\langle\omega,U^\intercal t\rangle=0$. Hence $R$ is a local
	diffeomorphism. By scale invariance,
	$\Psi_\gamma\circ R(y,t,s)=F(y,t)
	\coloneqq\Psi_\gamma(\Theta_\gamma(y)-U^\intercal t)$ does not depend
	on $s$, and $F(y,0)=\psi_\gamma(y)$. Since
	$\Theta_\gamma(y)\in M_\gamma$, Lemma~\ref{lem:parity-functional}(iii)
	shows that $\partial_tF(y,0)=0$ for all $y$, and therefore
	$\partial_y\partial_tF(y_0,0)=0$. Moreover,
	$\partial_t^2F(y_0,0)=-\sum_i\gamma_ib_i^{-2}u_iu_i^\intercal$ is
	negative definite. Consequently,
	\[
	\mathrm{Hess}\,(\Psi_\gamma\circ R)(y_0,0,0)
	=\mathrm{Hess}\,\psi_\gamma(y_0)\oplus\partial_t^2F(y_0,0)\oplus(0),
	\]
	and since $\partial_sR(y_0,0,0)=b$, this is a Gram matrix of
	$\mathfrak h$ in which the last summand corresponds to the radical
	$\R b$. Thus
	$\bar{\mathfrak h}\cong\mathrm{Hess}\,\psi_\gamma(y_0)\oplus
	\partial_t^2F(y_0,0)$. As $\bar{\mathfrak h}$ is non-degenerate, so is
	$\mathrm{Hess}\,\psi_\gamma(y_0)$, and
	$\varepsilon(\bar{\mathfrak h})
	=(-1)^n\sgn\det\mathrm{Hess}\,\psi_\gamma(y_0)$. Together with
	\eqref{eq:parity-eps} this proves the lemma.
\end{proof}

The last ingredient is the following classical fact from degree theory.

\begin{lemma}[{Castro--Lazer \cite{CastroLazer1979}, see also
	\cite{McKenzie1984,Amann1982}}]
	\label{lem:parity-castro-lazer}
	Let $\phi\in C^1(\R^d)$ with $\phi(z)\to\infty$ as $|z|\to\infty$.
	Assume that $\nabla\phi$ has only finitely many zeros
	$z_1,\dots,z_k$, that $\phi$ is of class $C^2$ in a neighbourhood of
	each $z_j$, and that the Hessians $\mathrm{Hess}\,\phi(z_j)$ are
	non-singular. Then
	\[
	\sum_{j=1}^k\sgn\det\mathrm{Hess}\,\phi(z_j)=1.
	\]
\end{lemma}

\begin{proof}
	By \cite{CastroLazer1979} (see \cite{McKenzie1984} for an elementary
	proof), the Brouwer degree of $\nabla\phi$ on the ball $B_r$ with
	respect to the value $0$ equals $1$ for all sufficiently large $r$. By
	the additivity of the degree, it equals the sum of the local indices
	of $\nabla\phi$ at the points $z_j$, and the local index at a
	non-degenerate zero is $\sgn\det\mathrm{Hess}\,\phi(z_j)$; see, e.g.,
	\cite{Deimling1985}.
\end{proof}

\begin{proof}[Proof of Theorem~\ref{thm:parity2}]
	As explained above, we may assume $\|\gamma\|_1=1$. By
	Proposition~\ref{prop:regular-points}(ii),
	Lemma~\ref{lem:parity-slice}(ii) and Lemma~\ref{lem:parity-index},
	$\psi_\gamma$ has finitely many critical points
	$y_j=\Theta_\gamma^{-1}(b^{(j)})$, where
	$S(U,\gamma)=\{b^{(1)},\dots,b^{(k)}\}$, and $\psi_\gamma$ is smooth
	with non-singular Hessian near each of them.

	If $d=0$, then $\mathcal Y$ is a point and $k=1$. If $d\geq1$, let
	$\sigma:\R^d\to\mathcal Y$ be a diffeomorphism; such a map exists
	since $\mathcal Y$ is an open convex subset of a $d$-dimensional
	affine space. Then $\phi=\psi_\gamma\circ\sigma$ satisfies the
	assumptions of Lemma~\ref{lem:parity-castro-lazer}: it tends to
	infinity at infinity by Lemma~\ref{lem:parity-coercive}, and at a
	critical point
	$\mathrm{Hess}\,\phi
	=D\sigma^\intercal\,(\mathrm{Hess}\,\psi_\gamma)\,D\sigma$. Hence, in
	both cases,
	$\sum_{j=1}^k\sgn\det\mathrm{Hess}\,\psi_\gamma(y_j)=1$, and
	Lemma~\ref{lem:parity-index} gives
	\[
	\sum_{j=1}^k\sgn\det Df_{\cv}^U(b^{(j)})=(-1)^{m+n-1}.
	\]
	A sum of $k$ terms $\pm1$ equals $\pm1$ only if $k$ is odd.
\end{proof}

\begin{remark}
	\label{rem:parity-sharp}
	Irreducibility of $U$ is not sufficient for
	Theorem~\ref{thm:parity}. If $U\in\U(2,4)$ is the set of outer unit
	normals of a trapezoid, then $U$ is irreducible, but by
	\cite[Theorem~2.14]{Subspace_Concentration_of_Geometric_Measures}
	(cf.~the proof of Theorem~\ref{thm:4goncase}) there is a non-empty
	open set of $\gamma$ with $|S(U,\gamma)|=2$, and it contains points
	$\gamma\notin B(U)$. In the proof above it is
	Lemma~\ref{lem:parity-coercive} that fails: normalized trapezoids may
	degenerate to segments, and $\psi_\gamma$ is no longer proper.
\end{remark}

\begin{remark}
	\label{rem:parity-consequences}
	\begin{enumerate}
		\item General position is used only via the uniform bound on the
		diameters of the polytopes $P(U,b)$ of volume one. Hence
		Theorem~\ref{thm:parity} holds for every $U\in\U(n,m)$ with this
		property. Such a $U$ is automatically irreducible, since for a
		reducible $U$ and $P(U,b)=Q_1\oplus Q_2$ the polytopes
		$t^{n_2}Q_1\oplus t^{-n_1}Q_2$, $t>0$, all have the same volume
		but unbounded diameters.

		\item Theorem~\ref{thm:parity} complements
		Theorem~\ref{thm:con}: the signed number of solutions is constant
		on all of $\R^m_{>0}\setminus B(U)$, and when $\gamma$ crosses
		$B(U)$, the cardinality of $S(U,\gamma)$ changes by an even
		number, as in the transitions between one and three solutions in
		Table~\ref{tab:liu-quadrilateral}.

		\item Lemmas~\ref{lem:parity-slice} and
		\ref{lem:parity-coercive} show that $\psi_\gamma$ attains its
		minimum, which gives a solution for every $\gamma\in\R^m_{>0}$.
		This reproves Zhu's existence theorem \cite{Zhu2014}
		(cf.~\cite[Theorem~III]{baumbach2025polynomialinequalitiesconevolumespolytopes});
		for existence results beyond general position see
		\cite{BoeroeczkyHegeduesZhu2016}.

		\item Let $U\in\U(2,m)$, $m\geq4$, be in general position. The
		proof of Theorem~\ref{thm:1.7} provides
		$\gamma\in C_{\cv}(U)\setminus B(U)$ with $|S(U,\gamma)|\geq2$,
		and thus, by Theorem~\ref{thm:parity}, $|S(U,\gamma)|\geq3$.

		\item The formula in Theorem~\ref{thm:inserting-normal},
		$\det DF(\widehat b^{(k)})
		=-\frac{\widehat b^{(k)}_{m+1}}{2\alpha\beta}
		\det Df_{\cv}^U(b^{(k)})$, is consistent with
		Theorem~\ref{thm:parity}: inserting a normal changes the sign
		$(-1)^{m+n-1}$, and so do the determinants at the solutions
		$b^{+,k}$ constructed there. If $U^+$ is in general position, the
		remaining solutions in $S(U^+,\gamma^+)$, if any, therefore have
		signed count zero.

		\item In Example~\ref{exp:pentagon} ($n=2$, $m=5$) the signs of
		$\det Df_{\cv}^U$ at the three solutions are $+,-,+$, in
		accordance with $(-1)^{m+n-1}=1$.
	\end{enumerate}
\end{remark}

\section{The derivative of the cone-volume map}
\label{sec:derivative}

In this final section we relate the differential structure of the
cone-volume map to the geometry of the prescribed outer unit normals.
We begin with an explicit formula for the derivative on a proper
type-cone.

We use the notation from Section~2 (cf.~\eqref{eq:nr1}, \eqref{eq:nr2}, \eqref{eq:facetineq}). In particular, $J^A(i)$ denotes
the set of indices of the facets adjacent to $F_i$, and
\[
F_i(b)
=
b_i u_i+
\left\{
x\in u_i^\perp:
\langle x,\nu_{i,j}\rangle\leq h_{i,j},
\ j\in J^A(i)
\right\},
\]
where
\[
h_{i,j}=b_j\csc\theta_{ij}-b_i\cot\theta_{ij}.
\]

\begin{theorem}
	\label{thm:dev}
	Let $U\in\U(n,m)$ and let $A\in\mathcal A(U)$ be a proper type-cone.
	Write $f_i=f_i^A$. Then, for $b\in\inter A$,
	\[
	\frac{\partial f_i}{\partial b_j}(b)=
	\begin{cases}
		\displaystyle
		\frac{1}{n(n-1)}
		\sum_{k\in J^A(i)}
		\vol_{n-2}\bigl(F_i(b)\cap F_k(b)\bigr)
		\bigl(b_k\csc\theta_{ik}-n b_i\cot\theta_{ik}\bigr),
		& j=i,\\[3mm]
		\displaystyle
		\frac{1}{n}
		\vol_{n-2}\bigl(F_i(b)\cap F_j(b)\bigr)
		\,b_i\csc\theta_{ij},
		& j\in J^A(i),\\[2mm]
		0,
		& \text{ otherwise}.
	\end{cases}
	\]
\end{theorem}

\begin{proof} 
	The standard variation formula for the volume of a polytope with
	respect to its support numbers gives
	\[
	\frac{\partial}{\partial h_{ij}}\vol_{n-1}(F_i(b))
	=
	\vol_{n-2}\bigl(F_i(b)\cap F_j(b)\bigr);
	\]
	see, e.g., \cite[Section~5.1]{Schneider2014}. Hence
	\[
	\frac{\partial}{\partial b_j}\vol_{n-1}(F_i(b))
	=
	\begin{cases}
		\displaystyle
		-\sum_{k\in J^A(i)}
		\vol_{n-2}\bigl(F_i(b)\cap F_k(b)\bigr)\cot\theta_{ik},
		& j=i,\\[2mm]
		\displaystyle
		\vol_{n-2}\bigl(F_i(b)\cap F_j(b)\bigr)\csc\theta_{ij},
		& j\in J^A(i),\\[2mm]
		0,
		& \text{ otherwise}.
	\end{cases}
	\]
	Since
	\[
	f_i(b)=\frac{b_i}{n}\vol_{n-1}(F_i(b)),
	\]
	the off-diagonal formulas follow immediately. For the diagonal term,
	we use the Minkowski formula in the $(n-1)$-dimensional affine hull of
	$F_i(b)$:
	\[
	(n-1)\vol_{n-1}(F_i(b))
	=
	\sum_{k\in J^A(i)}
	h_{ik}\,
	\vol_{n-2}\bigl(F_i(b)\cap F_k(b)\bigr).
	\]
	Substituting the expression for $h_{ik}$ gives
	\[
	\frac{\partial f_i}{\partial b_i}(b)
	=
	\frac{1}{n(n-1)}
	\sum_{k\in J^A(i)}
	\vol_{n-2}\bigl(F_i(b)\cap F_k(b)\bigr)
	\bigl(b_k\csc\theta_{ik}-n b_i\cot\theta_{ik}\bigr),
	\]
	as claimed.
\end{proof}

We next characterize when the cone-volume polynomial $f_i$ is linear in
the variable $b_i$. We use the following local form of a standard
consequence of the equality case of the Brunn--Minkowski inequality (cf.~\cite[Lemma~3.1]{henk2014cone}).

\begin{lemma}
	\label{lem:Facet_subset_relation}
	Let $P\subset\R^n$ be a polytope, $\dim P=n$, let
	$u\in\sph^{n-1}$, and for $t\in\R$ let 
	\[
	F(t)\coloneqq P\cap\{x\in\R^n:\langle u,x\rangle=t\}.
	\]
	Suppose that, for some $b\in\R$ and $\varepsilon>0$,
	\[
	\vol_{n-1}(F(b-\delta))
	=
	\vol_{n-1}(F(b))
	>0
	\qquad\text{for all }\delta\in[0,\varepsilon].
	\]
	Then $F(b-\delta)$ is a translate of $F(b)$ for every
	$\delta\in(0,\varepsilon)$.
\end{lemma}

\begin{proof}
	Set $H=\{x\in\R^n:\langle u,x\rangle=0\}$ and $K_t\coloneqq F(t)-tu\subset H$. Fix
	$\delta=\lambda\varepsilon$ with $\lambda\in(0,1)$. Convexity of $P$
	gives
	\[
	(1-\lambda)K_b+\lambda K_{b-\varepsilon}
	\subseteq K_{b-\delta}.
	\]
	All three sets have the same positive $(n-1)$-dimensional volume.
	Hence Brunn--Minkowski in $H$ gives equality, and therefore
	$K_b$ and $K_{b-\varepsilon}$ are homothetic. Since their volumes are
	equal, they are translates; write
	\[
	K_{b-\varepsilon}=K_b+v.
	\]
	It follows that $K_b+\lambda v\subseteq K_{b-\delta}$. The two sets
	again have the same volume, and hence
	\[
	K_{b-\delta}=K_b+\lambda v.
	\]
	Thus
	\[
	F(b-\delta)=F(b)+\lambda v-\delta u,
	\]
	which proves the assertion.
\end{proof}

\begin{proposition}
	\label{thm:gamma_linear_iff_zylinder}
	Let $U\in\U(n,m)$ and let $A\in\mathcal A(U)$ be a proper type-cone.
	Then $f_i^A$ has degree 1 in the variable $b_i$
	if and only if
	$
	\dim\lin \{u_j:j\in J^A(i)\} =n-1.$
\end{proposition}

\begin{proof}
	Suppose first that $f_i^A$ has degree 1 in $b_i$. Since
	\[
	f_i^A(b)=\frac{b_i}{n}\vol_{n-1}(F_i(b)),
	\]
	the polynomial $\vol_{n-1}(F_i(b))$ is independent of $b_i$.
	Fix $b\in\inter A$ and choose $\varepsilon>0$ such that
	$b-\delta e_i\in\inter A$ for $0\leq\delta\leq\varepsilon$.
	Since only the $i$-th support number changes,
	\[
	F_i(b-\delta e_i)
	=
	P(U,b)\cap\{x:\langle u_i,x\rangle=b_i-\delta\}.
	\]
	Hence Lemma~\ref{lem:Facet_subset_relation} gives
	$F_i(b-\delta e_i)=F_i(b)+t_\delta$ 
	for $0<\delta<\varepsilon$. Apparently,  $t_\delta\ne 0$ 
	and for every $j\in J^A(i)$ the $h_{F_i(b-\delta e_i)}(u_j)=b_j=h_{F_i(b)}(u_j)$, and therefore
	\[
	\langle u_j,t_\delta\rangle=0.
	\]
	Thus all normals adjacent to $u_i$ lie in the hyperplane
	$\{x\in\R^n : \langle t_\delta, x\rangle=0\}$.
	
	Conversely, suppose that the normals adjacent to $u_i$ are contained in
	an $(n-1)$-dimensional subspace $L$. Choose
	$v\ne 0$ orthogonal to $L$. As 
	$u_i\notin L$ we may assume that
	$\langle u_i,v\rangle=1$. Then for  $b\in\inter A$ and sufficiently small $s>0$ we get  
	\[
	F_i(b+s e_i)=F_i(b)+sv,
	\]
	because $\langle u_j,v\rangle=0$ for every $j\in J^A(i)$.
	Consequently, $\vol_{n-1}(F_i(b))$ is independent of $b_i$, and hence
	$f_i^A(b)$ has degree 1 in  $b_i$.
\end{proof}

\begin{example}
	\label{Example:ConvergencePconePartitionPconeU}
	Consider 
	\begin{equation*}
		U =
		\begin{pmatrix}
			0 & -\frac{1}{\sqrt{2}} & 0 & \frac{1}{\sqrt{2}} & 0 \\
			0 & 0 & \frac{1}{\sqrt{2}} & 0 & -\frac{2}{\sqrt{5}} \\
			1 & -\frac{1}{\sqrt{2}} & -\frac{1}{\sqrt{2}}
			& -\frac{1}{\sqrt{2}} & -\frac{1}{\sqrt{5}}
		\end{pmatrix}
		=(u_1,u_2,u_3,u_4,u_5)
	\end{equation*}    
	Let $A\in\mathcal A(U)$ be the proper type-cone in which $F_2(b)$ and $F_4(b)$ are triangles (see Figure~\ref{fig:linear-bi-example}). 
	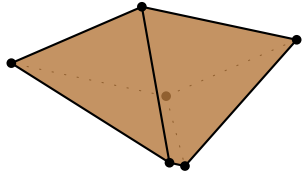
\begin{figure}[h]
		\centering
			\centering
			\begin{tikzpicture}%
				[x={(-0.574987cm, -0.248707cm)},
				y={(-0.818162cm, 0.174783cm)},
				z={(-0.000003cm, 0.952678cm)},
				scale=1.000000,
				back/.style={loosely dotted, thin},
				edge/.style={color=black!95!black, thick},
				facet/.style={fill=brown!95!black,fill opacity=0.800000},
				vertex/.style={inner sep=1pt,circle,draw=black!25!black,fill=black!75!black,thick}]
				%
				%
				\coordinate (1.00000, -0.50000, 1.00000) at (1.00000, -0.50000, 1.00000);
				\coordinate (1.00000, 2.00000, 1.00000) at (1.00000, 2.00000, 1.00000);
				\coordinate (-2.00000, 2.00000, 1.00000) at (-2.00000, 2.00000, 1.00000);
				\coordinate (-0.50000, 0.50000, -0.50000) at (-0.50000, 0.50000, -0.50000);
				\coordinate (-0.50000, 0.25000, -0.50000) at (-0.50000, 0.25000, -0.50000);
				\coordinate (-2.00000, -0.50000, 1.00000) at (-2.00000, -0.50000, 1.00000);
				\draw[edge,back] (1.00000, -0.50000, 1.00000) -- (1.00000, 2.00000, 1.00000);
				\draw[edge,back] (1.00000, -0.50000, 1.00000) -- (-0.50000, 0.25000, -0.50000);
				\draw[edge,back] (1.00000, -0.50000, 1.00000) -- (-2.00000, -0.50000, 1.00000);
				\node[vertex] at (1.00000, -0.50000, 1.00000)     {};
				\fill[facet] (-0.50000, 0.50000, -0.50000) -- (1.00000, 2.00000, 1.00000) -- (-2.00000, 2.00000, 1.00000) -- cycle {};
				\fill[facet] (-2.00000, -0.50000, 1.00000) -- (-2.00000, 2.00000, 1.00000) -- (-0.50000, 0.50000, -0.50000) -- (-0.50000, 0.25000, -0.50000) -- cycle {};
				\draw[edge] (1.00000, 2.00000, 1.00000) -- (-2.00000, 2.00000, 1.00000);
				\draw[edge] (1.00000, 2.00000, 1.00000) -- (-0.50000, 0.50000, -0.50000);
				\draw[edge] (-2.00000, 2.00000, 1.00000) -- (-0.50000, 0.50000, -0.50000);
				\draw[edge] (-2.00000, 2.00000, 1.00000) -- (-2.00000, -0.50000, 1.00000);
				\draw[edge] (-0.50000, 0.50000, -0.50000) -- (-0.50000, 0.25000, -0.50000);
				\draw[edge] (-0.50000, 0.25000, -0.50000) -- (-2.00000, -0.50000, 1.00000);
				\node[vertex] at (1.00000, 2.00000, 1.00000)     {};
				\node[vertex] at (-2.00000, 2.00000, 1.00000)     {};
				\node[vertex] at (-0.50000, 0.50000, -0.50000)     {};
				\node[vertex] at (-0.50000, 0.25000, -0.50000)     {};
				\node[vertex] at (-2.00000, -0.50000, 1.00000)     {};
			\end{tikzpicture}
		
		\caption{A polytope in the type-cone of Example~\ref{Example:ConvergencePconePartitionPconeU}.}
		\label{fig:linear-bi-example}
	\end{figure}
	
	The normals adjacent to $u_2$, namely $u_1,u_3$ and $u_5$, are contained in a two-dimensional subspace. Hence the cone-volume polynomial
	\[
	f_2(b_1,b_2,b_3,b_4,b_5)=\frac{1}{2\sqrt{2}}b_1^2b_2+\frac{2}{3}b_1b_2b_3+\frac{\sqrt{5}}{3\sqrt{2}}b_1b_2b_5+\frac{2\sqrt{2}}{9}b_2b_3^2+\frac{2\sqrt{5}}{9}b_2b_3b_5+\frac{5}{18\sqrt{2}}b_2b_5^2
	\]
	is linear in $b_2$. On the other hand, the normals adjacent to $u_3$, namely $u_1,u_2,u_4$ and $u_5$, are not contained in a two-dimensional subspace; hence the cone-volume polynomial
	\[
	\begin{split}
		f_3(b_1,b_2,b_3,b_4,b_5)=\frac{\sqrt{2}}{3}b_1^2b_3+\frac{2}{3}b_1b_2b_3+\frac{2}{3}b_1b_3b_4+\frac{4\sqrt{2}}{9}b_2b_3^2+\frac{2\sqrt{5}}{9}b_2b_3b_5\\
		-\frac{8\sqrt{2}}{27}b_3^3+\frac{4\sqrt{2}}{9}b_3^2b_4-\frac{8\sqrt{5}}{27}b_3^2b_5+\frac{2\sqrt{5}}{9}b_3b_4b_5-\frac{5\sqrt{2}}{27}b_3b_5^2
	\end{split}
	\]
	is not linear in $b_3$.
\end{example}

The preceding proposition describes one extreme. At the other extreme,
one may ask when $f_i^A$ contains the pure monomial $b_i^n$. Even
general position of the normals adjacent to $u_i$ does not suffice, as
the following example shows.

\begin{example}
	Let $U=(u_1,u_2,u_3,u_4,u_5)\in\U(3,5)$, where
	\begin{align*}
		u_1&=(0,0,1)^\top,\\
		u_2&=\frac{1}{\sqrt{4-2\sqrt2}}
		\left(0,-1,1-\sqrt2\right)^\top,\qquad
		u_3=\frac{1}{\sqrt{10}}(-1,0,-3)^\top,\\
		u_4&=\frac{1}{\sqrt{4-2\sqrt2}}
		\left(-\frac{\sqrt2}{2},\frac{\sqrt2}{2},1-\sqrt2\right)^\top,\qquad
		u_5=\frac{1}{\sqrt2}
		\left(\frac{\sqrt2}{2},\frac{\sqrt2}{2},-1\right)^\top .
	\end{align*}
	Consider a type-cone $A$ in which $u_1$ is adjacent to
	$u_2,u_3,u_4$ and $u_5$. The set $U$ is in general position, but
	the cone-volume polynomial
	\[
	\begin{split}
		f_{1}(b_1,b_2,b_3,b_4,b_5)=\frac{1}{3}\Big(
		&\sqrt{16+8\sqrt{2}}\,b_1^2b_2
		-2\sqrt{10}\,b_1^2b_3
		+\sqrt{16+8\sqrt{2}}\,b_1^2b_4
		+2\sqrt{2}\,b_1^2b_5\\
		&+(2-\sqrt{2})\,b_1b_2^2
		+\sqrt{40-20\sqrt{2}}\,b_1b_2b_3
		+\sqrt{16-8\sqrt{2}}\,b_1b_2b_5
		-5b_1b_3^2\\
		&+\sqrt{80-40\sqrt{2}}\,b_1b_3b_4
		+(-2+\sqrt{2})\,b_1b_4^2
		+\sqrt{8-4\sqrt{2}}\,b_1b_4b_5
		+b_1b_5^2\Big)
	\end{split}
	\]
	does not contain the monomial $b_1^3$.
\end{example}

The additional condition needed is that the type-cone is unbounded in
the direction $e_i$. Under this assumption, general position is required
only for the normals adjacent to $u_i$.

\begin{proposition}
	\label{ubd}
	Let $U\in\U(n,m)$ and let $A\in\mathcal A(U)$ be a proper type-cone.
	Assume that $\R_{\geq0}e_i\subseteq A$  
	and that the normals $\{u_j:j\in J^A(i)\}$
	are in general position. Then the cone-volume polynomial $f_i^A$
	contains the monomial $b_i^n$ with positive coefficient.
\end{proposition}

\begin{proof}
	Fix $\overline{b}\in\inter A$ and, for $\lambda\geq \overline{b}_i$, set
	\[
	b(\lambda)\coloneqq \overline{b}+(\lambda-\overline{b}_i)e_i,
	\qquad
	P_\lambda\coloneqq P(U,b(\lambda)).
	\]
	Since $\R_{\geq0}e_i\subseteq A$, we have
	$b(\lambda)\in\inter A$ for every $\lambda\geq \overline{b}_i$.
	
	Let $\beta$ denote the coefficient of $b_i^n$ in the polynomial $f_i^A(b)$. Since, for
	$j\neq i$, the polynomial $f_j^A$ is divisible by $b_j$, the monomial
	$b_i^n$ occurs in none of these polynomials. Hence, using
	$\vol_n(P(U,b(\lambda)))=\sum_{j=1}^m f_j^A(b(\lambda))$
	we obtain
	\[
	\beta
	=
	\lim_{\lambda\to\infty}
	\lambda^{-n}\vol_n(P_\lambda).
	\]
	Set
	\[
	Q\coloneqq \bigcap_{j\neq i}H^-_{u_j,\overline{b}_j},
	\qquad
	C\coloneqq \bigcap_{j\neq i}H^-_{u_j,0},
	\]
	where for $t\in\R$, $H^-_{u_j,t}=\{x\in\R^n :\langle u_j,x\rangle\leq t\}$. 
	Then $ P_\lambda=Q\cap H^-_{u_i,\lambda}$.  
	Since $u_i$ continues to define a facet of $P_\lambda$ for arbitrarily
	large $\lambda$, we have $h_Q(u_i)=\infty$, and hence
	$h_C(u_i)=\infty$.
	
	We claim that the recession cone $C$ is full-dimensional. Suppose otherwise and put
	\[
	V\coloneqq (\lin C)^\perp,
	\qquad d\coloneqq \dim V>0.
	\]
	The projection $Q\vert V$ is a $d$-dimensional bounded polytope in
	$V$, and therefore has at least $d+1$ facets. Let $v\in V$ be an outer
	unit normal of one of these facets. Since $v$ vanishes on $\lin C$, we have $F_Q(v)=F_Q(v)+C$,
	and, in particular, $F_Q(v)$ is a facet of $Q$.

	Moreover, since $h_C(u_i)=\infty$, for all sufficiently large
	$\lambda$ the truncation of $F_Q(v)$ by
	$H^-_{u_i,\lambda}$ is a facet of $P_\lambda$ adjacent to
	$F_i(P_\lambda)$. Thus $v$ is one of the normals
	$\{u_j:j\in J^A(i)\}$. Consequently, at least $d+1$ normals adjacent to $u_i$ lie in the
	$d$-dimensional subspace $V$, contradicting their general position.
	Hence
	\[
	\dim C=n.
	\]
	
	The set $C\cap H^-_{u_i,1}$ is bounded. Indeed, otherwise there exists a non-zero $w$ 
	satisfying 
	\[
	\langle u_j,w\rangle\leq0
	\qquad\text{for every }j=1,\ldots,m,
	\]
	which is impossible since $U$ positively spans $\R^n$. Since $C$ is
	full-dimensional, $C\cap H^-_{u_i,1}$ has positive $n$-dimensional
	volume. Moreover, $\overline{b}_j>0$ for all $j$, so $C\subseteq Q$ and
	\[
	P_\lambda
	\supseteq
	C\cap H^-_{u_i,\lambda}.
	\]
	By homogeneity of the cone $C$,
	\[
	\begin{aligned}
		\beta
		=
		\lim_{\lambda\to\infty}
		\lambda^{-n}\vol_n(P_\lambda)
		\geq
		\lim_{\lambda\to\infty}
		\lambda^{-n}
		\vol_n(C\cap H^-_{u_i,\lambda})
		=
		\vol_n(C\cap H^-_{u_i,1})
		>0.
	\end{aligned}
	\]
	Thus $f_i^A$ contains the monomial $b_i^n$ with positive coefficient.
\end{proof}

\section*{Acknowledgements}

The authors used ChatGPT (OpenAI) for language editing, stylistic improvements,
and for discussing the presentation of some arguments; no proof was generated by it. However,
Theorem~\ref{thm:parity} originated in an exchange with Claude (Anthropic):
the phenomenon had been observed numerically by the
authors beforehand, whereas the variational and degree-theoretic proof
given there, as well as a first draft of its exposition, were produced by
the model. All statements and proofs have been verified by the authors, who take full
responsibility for the mathematical content.

\newpage 

\appendix

\section{The quadrilateral classification of Liu et al.}
\label{app:liu-quadrilaterals}

For completeness, we record the part of the classification of Liu et
al.~\cite[Theorems~3.1 and~3.3]{LiuLuSunEtAl2024} that is used in the
proof of Theorem~\ref{thm:4goncase}.

Let $U=(u_1,u_2,u_3,u_4)\in\U(2,4)$ be in general position and let
$\gamma=(\gamma_1,\gamma_2,\gamma_3,\gamma_4)\in\R^4_{>0}$.
Liu et al. associate with $U$ parameters $\lambda,b,b_{12}$, where
$b_{12}\neq0$, $b>0$, $b\neq1$, and $\lambda<1$. To avoid confusion
with our support numbers, we write $\rho\coloneqq b$ and $\sigma\coloneqq b_{12}$.

If $\sigma>0$ (respectively, $\sigma<0$), then
$|S(U,\gamma)|$ equals the number of roots greater than $1$
(respectively, smaller than $1$) of
\begin{equation}
	\begin{split}
		q_\gamma(X)={}&
		\gamma_1\rho^2X^4
		-\rho\bigl(\gamma_1+\gamma_2
		+\rho(\gamma_1+\gamma_4)\bigr)X^3\\
		&+\rho(\gamma_1+2\gamma_2+\gamma_3+2\gamma_4)X^2\\
		&-\bigl(\gamma_3+\gamma_4
		+\rho(\gamma_2+\gamma_3)\bigr)X
		+\gamma_3.
	\end{split}
	\label{eq:liu-quartic}
\end{equation}
We shall only need the case
\[
\gamma_1\gamma_3=\gamma_2\gamma_4.
\]
For this case, the classification of Liu et al. is summarized in
Table~\ref{tab:liu-quadrilateral} 
where we use the abbreviation
\[
\lambda_0
\coloneqq 
\frac{(\gamma_3+\gamma_4)(\gamma_1+\gamma_2)^3}
{\gamma_1\gamma_2
	(\gamma_1+\gamma_2+\gamma_3+\gamma_4)^2}.
\]

\begin{table}[htbp]
	\centering
	\small
	\resizebox{\textwidth}{!}{%
		\begin{tabular}{|c|c|c|c|c|}
			\hline
			\multirow{2}*{Range of $\lambda$}
			& \multicolumn{2}{c|}{\multirow{2}*{Subcases}}
			& \multicolumn{2}{c|}{Number of solutions} \\ \cline{4-5}
			& \multicolumn{2}{c|}{}
			& $\sigma>0$ & $\sigma<0$ \\ \hline
			
			$(-\infty,-\frac{(\gamma_1-\gamma_2)^2}
			{4\gamma_1\gamma_2})$
			& \multicolumn{2}{c|}{None}
			& \multicolumn{2}{c|}{$1$} \\ \hline
			
			\multirow{4}*{
				$\{-\frac{(\gamma_1-\gamma_2)^2}
				{4\gamma_1\gamma_2}\}$}
			& \multirow{2}*{$\gamma_1<\gamma_2$}
			& $\lambda\neq\lambda_0$ & $2$ & \multirow{2}*{$1$} \\ \cline{3-4}
			& & $\lambda=\lambda_0$ & $1$ & \\ \cline{2-5}
			& \multirow{2}*{$\gamma_1>\gamma_2$}
			& $\lambda\neq\lambda_0$ & \multirow{2}*{$1$} & $2$ \\ \cline{3-3}\cline{5-5}
			& & $\lambda=\lambda_0$ & & $1$ \\ \hline
			
			\multirow{4}*{
				$(-\frac{(\gamma_1-\gamma_2)^2}
				{4\gamma_1\gamma_2},0)$}
			& \multirow{2}*{
				$\lambda>\frac{\gamma_1-\gamma_2}{2\gamma_1}$}
			& $\lambda\neq\lambda_0$ & $3$ & \multirow{2}*{$1$} \\ \cline{3-4}
			& & $\lambda=\lambda_0$ & $2$ & \\ \cline{2-5}
			& \multirow{2}*{
				$\lambda<\frac{\gamma_1-\gamma_2}{2\gamma_1}$}
			& $\lambda\neq\lambda_0$ & \multirow{2}*{$1$} & $3$ \\ \cline{3-3}\cline{5-5}
			& & $\lambda=\lambda_0$ & & $2$ \\ \hline
			
			\multirow{4}*{
				$\left(
				0,
				\left(
				\frac{\gamma_1+\gamma_2-\gamma_3-\gamma_4}
				{\gamma_1+\gamma_2+\gamma_3+\gamma_4}
				\right)^2
				\right)$}
			& \multirow{2}*{
				$\gamma_1+\gamma_2>\gamma_3+\gamma_4$}
			& $\lambda\neq\lambda_0$ & $3$ & \multirow{2}*{$1$} \\ \cline{3-4}
			& & $\lambda=\lambda_0$ & $2$ & \\ \cline{2-5}
			& \multirow{2}*{
				$\gamma_1+\gamma_2<\gamma_3+\gamma_4$}
			& $\lambda\neq\lambda_0$ & \multirow{2}*{$1$} & $3$ \\ \cline{3-3}\cline{5-5}
			& & $\lambda=\lambda_0$ & & $2$ \\ \hline
			
			\multirow{4}*{
				$\left\{
				\left(
				\frac{\gamma_1+\gamma_2-\gamma_3-\gamma_4}
				{\gamma_1+\gamma_2+\gamma_3+\gamma_4}
				\right)^2
				\right\}$}
			& \multirow{2}*{
				$\gamma_1+\gamma_2>\gamma_3+\gamma_4$}
			& $\lambda\neq\lambda_0$ & $2$ & \multirow{2}*{$1$} \\ \cline{3-4}
			& & $\lambda=\lambda_0$ & $1$ & \\ \cline{2-5}
			& \multirow{2}*{
				$\gamma_1+\gamma_2<\gamma_3+\gamma_4$}
			& $\lambda\neq\lambda_0$ & \multirow{2}*{$1$} & $2$ \\ \cline{3-3}\cline{5-5}
			& & $\lambda=\lambda_0$ & & $1$ \\ \hline
			
			$\left(
			\left(
			\frac{\gamma_1+\gamma_2-\gamma_3-\gamma_4}
			{\gamma_1+\gamma_2+\gamma_3+\gamma_4}
			\right)^2,
			1
			\right)$
			& \multicolumn{2}{c|}{None}
			& \multicolumn{2}{c|}{$1$} \\ \hline
		\end{tabular}%
	}
	\caption{The case distinction of Liu et al. in the case
		$\gamma_1\gamma_3=\gamma_2\gamma_4$.}
	\label{tab:liu-quadrilateral}
\end{table}

The only consequence of this classification needed in the main text
is the following.

\begin{proposition}
	\label{prop:liu-open-nonuniqueness}
	Let $U\in\U(2,4)$ be in general position. Then there exists a
	nonempty open set $\Omega\subset C_{\cv}(U)$ such that
	\[
	|S(U,\gamma)|=3
	\qquad\text{for every }\gamma\in\Omega.
	\]
\end{proposition}

\begin{proof}
	The case distinction in Table~\ref{tab:liu-quadrilateral} provides
	a vector $\gamma^{(0)}\in C_{\cv}(U)$ for which
	$|S(U,\gamma^{(0)})|=3$. For example, if $0<\lambda<1$ and
	$\sigma>0$, choose $\gamma_1,\gamma_3>0$ and set
	\[
	\gamma_4=\frac{\gamma_1\gamma_3}{\gamma_2}.
	\]
	As $\gamma_2\to\infty$,
	\[
	\left(
	\frac{\gamma_1+\gamma_2-\gamma_3-\gamma_4}
	{\gamma_1+\gamma_2+\gamma_3+\gamma_4}
	\right)^2
	\longrightarrow1,
	\]
	while the exceptional equality $\lambda=\lambda_0$ can be avoided.
	Hence, for sufficiently large $\gamma_2$, one of the
	three-solution cases in the table applies. The other parameter
	ranges are treated analogously.
	
	It remains to observe that the three solutions persist under small
	perturbations of $\gamma^{(0)}$. Suppose, for instance, that
	$\sigma>0$. By \eqref{eq:liu-quartic}, they correspond to three
	distinct roots of $q_{\gamma^{(0)}}$ in $(1,\infty)$. Moreover,
	\[
	q_\gamma(0)=\gamma_3>0,
	\qquad
	q_\gamma(1)=-(\rho-1)^2\gamma_4<0,
	\]
	so there is a fourth root in $(0,1)$. Thus all four roots are
	simple. Since simple roots depend continuously on the coefficients,
	the three roots in $(1,\infty)$ persist for all $\gamma$ in a
	sufficiently small neighbourhood of $\gamma^{(0)}$. The case
	$\sigma<0$ is analogous. Shrinking the neighbourhood if necessary
	yields a nonempty open set
	$\Omega\subset C_{\cv}(U)$ on which $|S(U,\gamma)|=3$.
\end{proof}

\printbibliography

@article{KolesnikovMilman2022,
  author  = {Kolesnikov, Alexander V. and Milman, Emanuel},
  title   = {Local {$L^p$}-{B}runn--{M}inkowski Inequalities for {$p<1$}},
  journal = {Memoirs of the American Mathematical Society},
  volume  = {277},
  number  = {1360},
  pages   = {v+78},
  year    = {2022},
  publisher = {American Mathematical Society},
  address = {Providence, RI},
  doi     = {10.1090/memo/1360},
  eprint  = {1711.01089},
  archivePrefix = {arXiv}
}

@book{BoeroeczkyFigalliRamos2026,
  author    = {Böröczky, Károly J. and Figalli, Alessio and Ramos, João P. G.},
  title     = {Isoperimetric Inequalities, Brunn--Minkowski Theory and Minkowski-Type Monge--Ampère Equations on the Sphere},
  series    = {Zürich Lectures in Advanced Mathematics},
  volume    = {33},
  publisher = {EMS Press},
  year      = {2026},
  isbn      = {978-3-98547-104-1},
  doi       = {10.4171/ZLAM/33}
}

@article{EleftheriouSavatovsky2020,
  author  = {Eleftheriou, Pantelis E. and Savatovsky, Alex},
  title   = {Expansions of real closed fields that introduce no new smooth functions},
  journal = {Annals of Pure and Applied Logic},
  volume  = {171},
  number  = {7},
  pages   = {102808},
  year    = {2020},
  issn    = {0168-0072},
  doi     = {10.1016/j.apal.2020.102808},
  eprint  = {1812.10151}
}

@book{Hatcher2002,
  author    = {Hatcher, Allen},
  title     = {Algebraic Topology},
  publisher = {Cambridge University Press},
  address   = {Cambridge},
  year      = {2002},
  isbn      = {0-521-79540-0}
}

@book{bochnak2013real, 
  author    = {Bochnak, J. and Coste, M. and Roy, M.-F.},
  title     = {Real algebraic geometry},
  publisher = {Springer},
  date      = {2013},
  isbn      = {978-3-662-03718-8}, 
  doi       = {10.1007/978-3-662-03718-8},
}

@article{ho1975note,
  title={A note on proper maps},
  author={C. Ho},
  journal={Proceedings of the American Mathematical Society},
  volume={51},
  number={1},
  pages={237--241},
  year={1975}
}

@misc{baumbach2025polynomialinequalitiesconevolumespolytopes,
      title={On polynomial inequalities for cone-volumes of polytopes}, 
      author={T. Baumbach and M. Henk},
      year={2025},
      eprint={2506.15370},
      archivePrefix={arXiv},
      primaryClass={math.MG},
      %url={https://arxiv.org/abs/2506.15370}, 
}

@Article{Zhu2014,
  author    = {Zhu, G.},
  journal   = {Advances in Mathematics},
  title     = {The logarithmic {M}inkowski problem for polytopes},
  date      = {2014},
  issn      = {0001-8708},
  pages     = {909--931},
  volume    = {262},
  doi       = {10.1016/j.aim.2014.06.004},
  %file      = {:Zhu2014 - The Logarithmic Minkowski Problem for Polytopes.pdf:PDF},
  groups    = {Convex Geometry, Lp-Brunn-Minkowski},
  keywords  = {52A40},
  mrnumber  = {3228445},
  owner     = {hugobib},
  timestamp = {2021-01-13},
}

@phdthesis{Subspace_Concentration_of_Geometric_Measures,
  author    = {Pollehn, H.},
  title     = {Subspace Concentration of Geometric Measures},
  date      = {2019-02}, 
  type      = {PhD thesis},
  school    = {Technische Universit\"at Berlin}
}

@article{henk2014cone,
  author    = {Henk, M. and Linke, E.},
  title     = {Cone-volume measures of polytopes},
  journal   = {Advances in Mathematics},
  volume    = {253},
  date      = {2014},
  pages     = {50--62},
  issn = {0001-8708}, 
  doi = {10.1016/j.aim.2013.11.015}, 
}

@Book{Schneider2014,
  author    = {Schneider, R.},
  publisher = {Cambridge University Press, Cambridge},
  title     = {Convex bodies: the {B}runn-{M}inkowski theory},
  date      = {2014},
  edition   = {expanded},
  isbn      = {978-1-107-60101-7},
  series    = {Encyclopedia of Mathematics and its Applications},
  volume    = {151},
  mrnumber  = {3155183},
  pages     = {xxii+736},
  doi       = {10.1017/CBO9781139003858},
}

@article{boroczky2013logarithmic_JAMS, 
  author    = {Böröczky, K. J. and Lutwak, E. and Yang, D. and Zhang, G.},
  title     = {The logarithmic Minkowski problem},
  journal   = {Journal of the American Mathematical Society},
  volume    = {26},
  number    = {3},
  date      = {2013},
  pages     = {831--852},
  issn      = {0894-0347}, 
  doi       = {10.1090/S0894-0347-2012-00741-3},
}

@Article{HuangYangZhang2025,
  author           = {Huang, Y. and Yang, D. and Zhang, G.},
  title            = {Minkowski Problems for Geometric Measures},
  date             = {2025},
   journal          = {Bulletin of the American Mathematical Society},
  date             = {2025},
  pages            = {359-425},
  volume           = {62},
  doi = {10.1090/bull/1858}
}

@Article{LiuLuSunEtAl2024,
  author           = {Liu, Y. and Lu, X. and Sun, Q. and Xiong, G.},
  journal          = {Pure and Applied Mathematics Quarterly},
  title            = {The logarithmic Minkowski problem in $\R^2$},
  date             = {2024},
  number           = {2},
  pages            = {869--902},
  volume           = {20},
  creationdate     = {2024-05-06T14:31:35},
  modificationdate = {2024-05-06T14:33:40},
doi = {10.4310/PAMQ.2024.v20.n2.a5}
}

@Book{Aigner1997,
  author           = {Aigner, M.},
  publisher        = {Springer-Verlag, Berlin},
  title            = {Combinatorial theory},
  date             = {1997},
  isbn             = {3-540-61787-6},
  note             = {Reprint of the 1979 original},
  series           = {Classics in Mathematics},
  creationdate     = {2023-08-10T15:11:40},
  doi              = {10.1007/978-3-642-59101-3},
  modificationdate = {2023-08-10T15:12:05},
  mrnumber         = {1434477},
  pages            = {viii+483},
}

@InCollection{Boeroeczky2023,
  author           = {B{\"o}r{\"o}czky, K. J.},
  booktitle        = {Harmonic analysis and convexity},
  publisher        = {Berlin: De Gruyter},
  title            = {The logarithmic {Minkowski} conjecture and the {{\(L_p\)}}-{Minkowski} problem},
  date             = {2023},
  isbn             = {978-3-11-077537-2; 978-3-11-077538-9},
  pages            = {83--118},
  creationdate     = {2025-03-25T11:24:29},
  doi              = {10.1515/9783110775389-003},
  groups           = {Lp-Brunn-Minkowski},
  keywords         = {35J96,52A40},
  modificationdate = {2025-03-25T11:24:29},
  owner            = {hugobib},
  %url              = {real.mtak.hu/191895/1/2210.00194v3.pdf},
  zbl              = {1525.35156},
  zbmath           = {7756042},
}

@article{979f1be4-becf-36f2-be30-0bebf5263aa4,
 ISSN = {0003486X, 19398980},
 URL = {http://www.jstor.org/stable/1968544},
 author = {A. P. Morse},
 journal = {Annals of Mathematics},
 number = {1},
 pages = {62--70},
 publisher = {[Annals of Mathematics, Trustees of Princeton University on Behalf of the Annals of Mathematics, Mathematics Department, Princeton University]},
 title = {The Behavior of a Function on Its Critical Set},
 volume = {40},
 year = {1939}
}

@Article{ChenLiZhu2019,
  author           = {Chen, S. and Li, Q. and Zhu, G.},
  journal          = {Transactions of the American Mathematical Society},
  title            = {The logarithmic {M}inkowski problem for non-symmetric measures},
  date             = {2019},
  issn             = {0002-9947},
  number           = {4},
  pages            = {2623--2641},
  volume           = {371},
  doi              = {10.1090/tran/7499},
  modificationdate = {2021-12-30T18:45:51},
  mrnumber         = {3896091},
}

@Article{Milnor1964,
 ISSN = {00029939, 10886826},
 URL = {http://www.jstor.org/stable/2034050},
 author = {J. Milnor},
 journal = {Proceedings of the American Mathematical Society},
 number = {2},
 pages = {275--280},
 publisher = {American Mathematical Society},
 title = {On the Betti Numbers of Real Varieties},
 volume = {15},
 year = {1964}
}

@book{schlafli1901theorie,
  title={Theorie der vielfachen Kontinuit{\"a}t},
  author={Schl{\"a}fli, L. and Graf, J.H. and Schl{\"a}fli, L. and Schl{\"a}fli, L. and Graf, J.H. and Graf, J.H. and Schweizerische Naturforschende Gesellschaft},
  series={Denkschriften, Neue, d. allg. schweiz. Gesellschaft f. d. ges. Naturwiss},
  url={https://books.google.de/books?id=6PG-zwEACAAJ},
  year={1901},
  publisher={Z{\"u}rcher \& Furrer}
}

@article{boroczky2012log,
title = {The log-Brunn–Minkowski inequality},
journal = {Advances in Mathematics},
volume = {231},
number = {3},
pages = {1974-1997},
year = {2012},
issn = {0001-8708},
doi = {10.1016/j.aim.2012.07.015},
author = {K. J. Böröczky and Erwin Lutwak and Deane Yang and Gaoyong Zhang},
}

@misc{OSCAR,
  key          = {OSCAR},
  organization = {The OSCAR Team},
  title        = {O{SCAR} -- {O}pen {S}ource {C}omputer {A}lgebra {R}esearch system, {V}ersion 1.7.3},
  year         = {2026},
  url          = {https://www.oscar-system.org},
}

@book{OSCAR-book,
  editor = {Decker, W. and Eder, C. and Fieker, C. and Horn, M. and Joswig, M.},
  title = {The {C}omputer {A}lgebra {S}ystem {OSCAR}: {A}lgorithms and {E}xamples},
  year = {2025},
  publisher = {Springer},
  series = {Algorithms and {C}omputation in {M}athematics},
  volume = {32},
  edition = {1},
  %url = {https://link.springer.com/book/9783031621260},
  issn = {1431-1550},
  doi = {10.1007/978-3-031-62127-7},
}

@article{CHEN2020107166,
title = {The $L_p$-Brunn-Minkowski inequality for $p<1$},
journal = {Advances in Mathematics},
volume = {368},
pages = {107166},
year = {2020},
issn = {0001-8708},
doi = {10.1016/j.aim.2020.107166},
author = {S. Chen and Y. Huang and Q. Li and J. Liu}
}

@article{Stancu2002,
  author       = {Stancu, Alina},
  title        = {The discrete planar $L_0$-Minkowski problem},
  journaltitle = {Advances in Mathematics},
  volume       = {167},
  pages        = {160--174},
  year         = {2002},
}

@article{Stancu2003,
  author       = {Stancu, Alina},
  title        = {On the number of solutions to the discrete two-dimensional $L_0$-Minkowski problem},
  journaltitle = {Advances in Mathematics},
  volume       = {180},
  pages        = {290--323},
  year         = {2003},
  doi          = {10.1016/S0001-8708(03)00005-7},
}

@article{Yagisita2006,
  author       = {Yagisita, Hiroki},
  title        = {Non-uniqueness of self-similar shrinking curves for an anisotropic curvature flow},
  journaltitle = {Calculus of Variations and Partial Differential Equations},
  volume       = {26},
  number       = {1},
  pages        = {49--55},
  year         = {2006},
  doi          = {10.1007/s00526-005-0357-2},
}

@article{CastroLazer1979,
  author       = {Castro, Alfonso and Lazer, Alan C.},
  title        = {Critical point theory and the number of solutions of a nonlinear Dirichlet problem},
  journaltitle = {Annali di Matematica Pura ed Applicata},
  volume       = {120},
  pages        = {113--137},
  year         = {1979},
  doi          = {10.1007/BF02411940},
}

@article{McKenzie1984,
  author       = {McKenzie, Ross H.},
  title        = {The degree of the gradient of a coercive functional},
  journaltitle = {Journal of the Australian Mathematical Society. Series A},
  volume       = {37},
  number       = {3},
  pages        = {413--415},
  year         = {1984},
  doi          = {10.1017/S1446788700022369},
}

@article{Amann1982,
  author       = {Amann, Herbert},
  title        = {A note on degree theory for gradient mappings},
  journaltitle = {Proceedings of the American Mathematical Society},
  volume       = {85},
  number       = {4},
  pages        = {591--595},
  year         = {1982},
  doi          = {10.1090/S0002-9939-1982-0660610-2},
}

@book{Deimling1985,
  author    = {Deimling, Klaus},
  title     = {Nonlinear Functional Analysis},
  publisher = {Springer},
  location  = {Berlin},
  year      = {1985},
}

@article{BoeroeczkyHegeduesZhu2016,
  author       = {Böröczky, Károly J. and Hegedűs, Pál and Zhu, Guangxian},
  title        = {On the discrete logarithmic Minkowski problem},
  journaltitle = {International Mathematics Research Notices},
  volume       = {2016},
  number       = {6},
  pages        = {1807--1838},
  year         = {2016},
  doi          = {10.1093/imrn/rnv189},
}
\end{document}